\documentclass{amsart}
 \calclayout

\usepackage{bm}

\makeatletter
\DeclareRobustCommand*{\bfseries}{%
  \not@math@alphabet\bfseries\mathbf
  \fontseries\bfdefault\selectfont
  \boldmath
}
\makeatother

\newcommand*\refc{\eqref}    % use to put parentheses around referenced items
\theoremstyle{plain}
\newtheorem{theorem}{Theorem}[section]

\theoremstyle{plain}
\newtheorem{proposition}[theorem]{Proposition}

\theoremstyle{plain}
\newtheorem{corollary}[theorem]{Corollary}

\theoremstyle{plain}
\newtheorem{lemma}[theorem]{Lemma}

\theoremstyle{definition}
\newtheorem{definition}[theorem]{Definition}
\theoremstyle{definition}
\newtheorem{example}[theorem]{Example}

\theoremstyle{definition}
\newtheorem{remark}[theorem]{Remark}

\theoremstyle{definition}
\newtheorem{notation}[theorem]{Notation}

\theoremstyle{definition}
\newtheorem{conjecture}[theorem]{Conjecture}

\theoremstyle{plain}
\newtheorem{assumption}[theorem]{Assumption}

\numberwithin{equation}{section}
\numberwithin{figure}{section}
\def\classno{\subjclass[2000]}
\def\extraline{\thanks}

\address{ Department of Mathematics\\ University of Iowa\\ Iowa
City, Iowa}
\email{ frederick-goodman@uiowa.edu}

\def\ignore#1{\relax}

\usepackage[svgnames]{xcolor}
\usepackage[plainpages=false, pdfpagelabels,  colorlinks=true, pdfstartview=FitV, linkcolor=DarkBlue, citecolor=DarkRed, urlcolor=blue]{hyperref}

\def\Z{{\mathbb Z}}

\def\Q{{\mathbb Q}}

\def\inv{^{-1}}

\def\ppm{^{\pm 1}}

\def\abmw #1{\widehat{\mathcal W}_{#1}}  %stands for affine Birman-Wenzl
\def\affbmw{\abmw}
\def\bmw #1{\mathcal W_{#1}}  %stands for Birman-Wenzl
\def\degbmw#1{\mathcal N_{#1}}   %stands for degenerate Birman Wenzl
\def\affdegbmw#1{\widehat{\mathcal N}_{#1}}
\def\both#1{\mathcal A_{#1}}
\def\affboth#1{\widehat {\mathcal A}_{#1}}

\def\dfn{\mathcal D_{f,n}}

\def\la{\lambda}
\def\La{\Lambda}

\def\char{{\rm char}}

\def\abold{{\bm a}}
\def\ubold{{\bm u}}
\def\qbold{{\bm q}}

\def\rhobold{{\bm \rho}}
\def\omegabold{{\bm \omega}}

 \def\S{{\mathfrak{S}}}

\title[Cyclotomic BMW algebras]{Remarks on Cyclotomic and Degenerate Cyclotomic BMW algebras}

\author{Frederick M. Goodman}

\classno{20C08, 16G99, 81R50}
\extraline{This work grew out of discussions with Arun Ram regarding the meaning of admissibility conditions.   I am grateful to Arun also for hospitality and financial support at the University of Melbourne.  I am grateful to  Kevin Buzzard for  responding  to a question on {\tt mathoverflow.net}.  I thank the referee for some helpful suggestions.}

\begin{document}
\noindent
{\em Journal of Algebra}  {\bf 364} (2012) 13-37
\bigskip
  \maketitle
 
 \begin{abstract}  
We relate the structure of cyclotomic and degenerate cyclotomic BMW algebras, for arbitrary parameter values,   to that for admissible parameter values.  In particular, we show that these algebras are cellular.
We characterize those parameter sets for affine  BMW algebras over an algebraically closed field that permit the algebras to have non--trivial cyclotomic quotients.
 \end{abstract}

 \section{Introduction}
 This paper is a contribution to the study of affine and degenerate affine Birman--Wenzl--Murakami (BMW) algebras.
 In order to study the finite dimensional representation theory of these infinite dimensional algebras, one introduces cyclotomic quotients, which are BMW analogues of cyclotomic and degenerate cyclotomic Hecke algebras (see ~\cite{ariki-koike,ariki-book, kleshchev-book}).  
 
 A peculiar feature of the cyclotomic algebras is that the parameters cannot be chosen arbitrarily; that is, unless the parameters satisfy certain relations, the algebras (defined over a field) collapse to cyclotomic or degenerate cyclotomic Hecke algebras.  These ``obligatory" conditions did not seem adequate at first to develop the representation theory.  Consequently,  several authors, notably Ariki, Mathas and Rui ~\cite{ariki-mathas-rui},  Wilcox and Yu ~\cite{Wilcox-Yu},  and Rui and Xu ~\cite{rui-2008} introduced stronger ``admissibility" conditions under which the algebras could be shown to have a well--behaved representation theory.
 
Up until now, the cyclotomic algebras have been studied only under the assumption of admissibility of the parameters.  Despite the successes achieved, this was not satisfactory, since {\em a priori} the admissibility requirement might be too restrictive to capture the entire finite dimensional representation theory of the affine algebras.  

In this paper, we extend the analysis of cyclotomic and degenerate cyclotomic BMW algebras to include the case of non-admissible parameters. We show that the structure and representation theory of the cyclotomic algebras with non-admissible parameters can be derived from that of the algebras with admissible parameters.

 \subsection{Background}
 Affine and cyclotomic BMW algebras and their degenerate versions  arise naturally by several different ``affinization" processes.  One such process amounts to making the Jucys-Murphy elements in the ordinary BMW or Brauer algebras into variables, retaining the relations between these elements and the standard generators of the BMW or Brauer algebras.
 This point of view was stressed by Nazarov, in defining degenerate affine BMW algebras~\cite{nazarov}.  
 For the BMW algebras, there is a geometric affinization process:    The ordinary BMW algebras can be realized as algebras of tangles in the disc cross the interval, modulo Kauffman skein relations~\cite{morton-wassermann}.  To affinize these algebras, one should replace the disc by the annulus;  alternatively, one replaces the ordinary braid group by the affine or type $B$  braid group.
 This is the motivation cited by H\"aring--Oldenburg~\cite{H-O2} for introducing affine and cyclotomic BMW algebras.   Finally,   Orellana and Ram provide an affinization of Schur--Weyl duality~\cite{orellana-ram} which produces representations of the affine braid group by $\check R$--matrices of a quantum group;  for symplectic or orthogonal quantum groups, this process yields representations of cyclotomic BMW algebras (over the complex numbers, with special parameters).   \ignore{The classical version of this construction yields representations of 
 degenerate cyclotomic BMW algebras.}
 
 As mentioned above, degenerate affine BMW algebras were introduced by Nazarov  \cite{nazarov}  under the name {\em affine Wenzl algebras}.    
The   cyclotomic  quotients of these algebras were introduced by Ariki, Mathas, and Rui in \cite{ariki-mathas-rui}  and studied further by Rui and Si in~\cite{rui-si-degenerate}, under the name {\em cyclotomic Nazarov--Wenzl algebras}.
  Affine and cyclotomic BMW algebras were introduced by 
 H\"aring--Oldenburg in~\cite{H-O2} and    studied by 
Goodman and   Hauschild  Mosley~\cite{GH1,GH2,GH3,goodman-2008},  Rui, Xu, and Si~\cite{rui-2008,Rui-Si-Cyclotomic-II},   Wilcox and Yu~\cite{Wilcox-Yu,Wilcox-Yu2,Wilcox-Yu3,Yu-thesis}, and Ram, Orellana, Daugherty and Virk~\cite{orellana-ram, Ram-Daugherty-Virk}.

The papers cited above study the algebras under the assumption of admissibility.    It has been shown that the algebras with admissible parameters are cellular
~\cite{ariki-mathas-rui,Wilcox-Yu2,Yu-thesis,rui-2008,Rui-Si-Cyclotomic-II,goodman-2008,goodman-graber1,goodman-graber2};  simple modules over a field have been classified~\cite{rui-si-degenerate,
Rui-Si-Cyclotomic-II};
and the non-degenerate cyclotomic BMW algebras have been shown to be isomorphic to algebras of tangles~\cite{GH2,GH3,Wilcox-Yu2,Wilcox-Yu3,Yu-thesis}.   

\ignore{Despite the success achieved so far in the representation theory, it was not entirely satisfactory that the theory could be worked out only under the assumption of admissibility.  It is reasonable to regard the affine and degenerate affine BMW algebras as the primary objects of interest;  finite dimensional representations of the affine algebras factor through cyclotomic quotients, but conceivably through cyclotomic quotients with non--admissible parameters.
On the other hand, there was 
 already some evidence in the literature  that the admissible case should govern the general case;  see, for example,  the theorem of Ariki, Mathas and Rui, quoted here as Theorem \ref{theorem:  ARM theorem on deg. affine algebras with rational coefficient sequence}.}

\subsection{Summary of results}
In this note, we show that  the structure of the cyclotomic and degenerate cyclotomic BMW  algebras for general parameters can be  derived from the admissible case.  An affine (resp.\  degenerate affine) BMW algebra $A_n$  contains a copy of the finite dimensional BMW algebra (resp.\  Brauer algebra)  $B_n$ and an additional ``affine" generator $y_1$, satisfying several relations with the generators of $B_n$.  A cyclotomic quotient is determined by a polynomial relation
\begin{equation}\label{introduction cyclotomic relation}
(y_1 - u_1) \cdots (y_1 - u_r) = 0.
\end{equation}
Denote the cyclotomic quotient determined by \refc{introduction cyclotomic relation} by $A_{n, r}(u_1, \dots, u_r)$.  Let $J_{n, r}(u_1, \dots, u_r)$  denote the ideal generated by the ``contraction" $e_1$ in $A_{n, r}(u_1, \dots, u_r)$.  Then we have a short exact sequence
\begin{equation}\label{introduction short exact sequence 1}
0 \to J_{n, r}(u_1, \dots, u_r) \to A_{n, r}(u_1, \dots, u_r) \to H_n(u_1, \dots, u_r) \to 0,
\end{equation}
where $H_n(u_1, \dots, u_r)$ is the cyclotomic Hecke algebra (resp.\  degenerate cyclotomic Hecke algebra).
Admissibility of the parameters means that $\{e_1, y_1e_1, \dots, y_1^{r-1} e_1\}$ is linearly independent in
\break $A_{2,r}(u_1, \dots, u_r)$;  this condition translates into specific conditions on the parameters of the algebra which are discussed in Sections \ref{section: admissibility in the degenerate case} and  \ref{subsection: admissibility in the non degenerate case}.  Suppose now that we are working over a field and that admissibility fails, but $e_1 \ne 0$;  then there exists a $d$ with
$0 < d < r$ such that $\{e_1, y_1e_1, \dots, y_1^{d-1} e_1\}$ is linearly independent in $A_{2,r}(u_1, \dots, u_r)$ but
$\{e_1, y_1e_1, \dots, y_1^{d} e_1\}$ is linearly dependent. We say that the parameters are $d$--semi--admissible.  We show that there exists a subset $\{v_1, \dots, v_d\} \subset \{u_1, \dots, u_r\}$ such that
\begin{enumerate}
\item $A_{n, d}(v_1, \dots, v_d)$ has admissible parameters, and
\item $J_{n, d}(v_1, \dots, v_d) \cong J_{n, r}(u_1, \dots, u_r)$.
\end{enumerate}
Thus we have
\begin{equation}\label{introduction short exact sequence 2}
0 \to J_{n, d}(v_1, \dots, v_d) \to A_{n, r}(u_1, \dots, u_r) \to H_n(u_1, \dots, u_r) \to 0.
\end{equation}
Two consequences of this analysis are the following:
\begin{enumerate}
\item The cyclotomic algebras are  cellular, under very mild hypotheses; in particular, when the ground ring is a field, the algebras are always cellular.  
\item Every finite dimensional simple module of an affine
(resp.\  degenerate affine)
  BMW algebra   over an algebraically closed field factors through a cyclotomic 
  (resp.\  degenerate cyclotomic)  
   BMW algebra with admissible parameters, or through a cyclotomic (resp.\  degenerate cyclotomic) Hecke algebra.
\end{enumerate}
 The latter result is a step towards classifying the simple modules of the affine and degenerate affine BMW algebras over an algebraically closed field.
   
   The main results of Ariki, Mathas and Rui~\cite{ariki-mathas-rui} regarding degenerate cyclotomic BMW algebras depend on the hypothesis that $2$ is invertible in the ground ring.  We point out in  that this hypothesis can be eliminated; see Section 3.

Finally, we characterize those parameter sets for affine  BMW algebras over an algebraically closed field that permit the algebras to have non--trivial cyclotomic quotients, or equivalently,  finite dimensional modules $M$ with $e_1M \ne 0$; see Theorem \ref{theorem:  theorem on  affine bmw algebras with rational coefficient sequence}.  
The analogous result for degenerate affine BMW algebras was proved in~\cite{ariki-mathas-rui}; we have made a minor improvement by removing the restriction that the characteristic of the field should  be different from $2$;  see Theorem   \ref{theorem:  ARM theorem on deg. affine algebras with rational coefficient sequence}.

 \section{Preliminaries}

\subsection{Definition of degenerate affine and cyclotomic BMW algebras}

Fix a positive integer $n$ and a commutative ring $S$ with
multiplicative identity.  Let  $\Omega=\{\omega_a: \ a\ge0\} $  be a sequence of elements of $S$.

\begin{definition}[\cite{nazarov}]
\label{Waff relations}
 The  {\em degenerate affine BMW algebra}
$\affdegbmw {n, S} =\affdegbmw {n, S}(\Omega)$ is 
the unital associative $S$--algebra with generators
$\{s_i,e_i, y_j : \ 1\le i<n \text{ and }1\le j\le n \} $
and relations:
\begin{enumerate}
    \item (Involutions)\ \ 
$s_i^2=1$, for $1\le i<n$.
    \item (Affine braid relations)
\begin{enumerate}
\item $s_is_j=s_js_i$ if $|i-j|>1$,
\item $s_is_{i+1}s_i=s_{i+1}s_is_{i+1}$,  for $1\le i<n-1$,
\item $s_iy_j=y_js_i$ if $j\ne i,i+1$.
\end{enumerate}
    \item (Idempotent relations)\ \ 
$e_i^2=\omega_0e_i$, \ \  for $1\le i<n$.
   \item (Compression relations)\ \ 
        $e_1y_1^ae_1=\omega_ae_1$, for $a>0$.
    \item (Commutation relations)
\begin{enumerate}
\item $s_ie_j=e_js_i$, and  $e_ie_j=e_je_i$ if $|i-j|>1$,
%\item $e_ie_j=e_je_i$, if $|i-j|>1$,
\item $e_iy_j=y_je_i$,  if $j\ne i,i+1$,
\item $y_iy_j=y_jy_i$,  for $1\le i,j\le n$.
\end{enumerate}
  
    \item (Tangle relations)
\begin{enumerate}
\item $e_is_i=e_i=s_ie_i$,  for $1\le i\le n-1$,
\item $s_ie_{i+1}e_i=s_{i+1}e_i$, and  $e_ie_{i+1}s_i=e_i s_{i+1}$,  for $1\le i\le n-2$,
\item $e_{i+1}e_is_{i+1}=e_{i+1}s_i$,  and $s_{i+1}e_ie_{i+1}=s_i e_{i+1}$,  for $1\le i\le n-2$.
\item  $e_{i+1}e_ie_{i+1}=e_{i+1}$,  and 
        $e_ie_{i+1}e_i=e_i$, for  % 
        $1\le i\le n-2$.
\end{enumerate}
  \item (Skein relations)\ \ 
        $s_iy_i-y_{i+1}s_i=e_i-1$,  and\ 
        $y_is_i-s_iy_{i+1}=e_i-1$,\  for % 
        $1\le i<n$.

    \item (Anti--symmetry relations)\ \ 
        $e_i(y_i+y_{i+1})=0$,  and         $(y_i+y_{i+1})e_i=0$, for %  
        $1\le i<n$.
\end{enumerate}
\end{definition}

\begin{definition}[\cite{ariki-mathas-rui}] \label{cyclo NZ}
Fix an integer $r\ge1$ and elements $u_1,\dots,u_r $ in $S$.     
The {\em degenerate cyclotomic  BMW  algebra} $\degbmw{n, S, r}=\degbmw{n, S, r}(\Omega; u_1, \dots, u_r)$ is the quotient of the  degenerate affine BMW algebra $\affdegbmw {n, S}(\Omega)$ by the 
cyclotomic relation  $(y_1-u_1)\dots(y_1-u_r) = 0.$
\end{definition}

\medskip
Note that, due to the symmetry of the relations,  $\affdegbmw {n, S}$  has a unique $S$--linear  algebra involution $*$
(that is, an algebra anti-automorphism of order $2$)    
 such that $e_i^* = e_i$,  $s_i^* = s_i$,  and $y_i^* = y_i$   for all $i$.  The involution passes to cyclotomic quotients.  

 \subsection{Definition of  affine and cyclotomic BMW algebras}
 
Fix an integer $n \ge 0$, and a commutative ring 
$S$  with invertible elements $\rho$ and $q$, and a sequence of elements $\Omega = (\omega_a)_{a \ge 0}$,  satisfying
\begin{equation} \label{equation:  basic relation in ground ring}
\rho\inv - \rho=   (q\inv -q) (\omega_0 - 1).
\end{equation}

\begin{definition}[\cite{H-O2}] \label{definition:  cyclotomic BMW}
The {\em affine BMW algebra}  $\abmw{n, S} = \abmw{n, S}(\rho, q, \Omega)$  is the unital associative $S$--algebra
with generators $y_1^{\pm 1}$, $g_i^{\pm 1}$  and
$e_i$ ($1 \le i \le n-1$) and relations:
\begin{enumerate}
\item (Inverses) $g_i g_i\inv = g_i\inv g_i = 1$ and 
$y_1 y_1\inv = y_1\inv y_1= 1$.

\item (Affine braid relations) 
\begin{enumerate}
\item $g_i g_{i+1} g_i = g_{i+1} g_ig_{i+1}$ and 
$g_i g_j = g_j g_i$ if $|i-j|  \ge 2$.
\item $y_1 g_1 y_1 g_1 = g_1 y_1 g_1 y_1$ and $y_1 g_j =
g_j y_1 $ if $j \ge 2$.
\end{enumerate}
\item (Idempotent relation) $e_i^2 = \omega_0 e_i$.
\item (Compression relations) For $j \ge 1$, $e_1 y_1^{ j} e_1 = \omega_j e_1$. 
\item (Commutation relations) 
\begin{enumerate}
\item $g_i e_j = e_j g_i$  and
$e_i e_j = e_j e_i$  if $|i-
j|
\ge 2$. 
\item $y_1 e_j = e_j y_1$ if $j \ge 2$.
\end{enumerate}
\item (Tangle relations)
\begin{enumerate}
\item $g_i e_i = e_i g_i = \rho \inv e_i$
 and $e_i g_{i \pm 1} e_i = \rho  e_i$.
\item $e_i e_{i\pm 1} e_i = e_i$,
\item $g_i g_{i\pm 1} e_i = e_{i\pm 1} e_i$ and
$ e_i  g_{i\pm 1} g_i=   e_ie_{i\pm 1}$. 
\end{enumerate}
\item (Kauffman skein relation)  $g_i - g_i\inv = (q - q\inv)(1- e_i)$.
\item (Unwrapping relation) $e_1 y_1 g_1 y_1 g_1 = e_1 = g_1 y_1 
g_1 y_1 e_1$.
\end{enumerate}
\end{definition}

 \begin{definition}[\cite{H-O2}] \label{cyclo BMW}
Fix an integer $r\ge1$ and invertible elements $u_1,\dots,u_r $ in $S$,  
The {\em  cyclotomic  BMW  algebra} $\bmw{n, S, r}=\bmw{n, S, r}(\rho, q, \Omega; u_1, \dots, u_r)$ is the quotient of the affine BMW algebra $\abmw{n, S}(\rho, q, \Omega)$ by the cyclotomic relation
$(y_1-u_1)\dots(y_1-u_r) = 0.$
\end{definition}
 
 \medskip
As in the degenerate case,  $\abmw{n, S}$  has a unique $S$--linear  algebra involution $*$
such that $e_i^* = e_i$ and  $g_i^* = g_i$,   for all $i$,   and $y_1^* = y_1$.  The involution passes to cyclotomic quotients.

 \subsection{Admissibility}
 
 \begin{notation}
 Let $\both{n, S, r}$  denote either the cyclotomic BMW algebra $\bmw{n,S, r}$ 
 (with parameters $\rho$, $q$,  $\Omega = (\omega_a)_{a \ge 0}$, and $u_1, \dots, u_r$)  or 
  the degenerate cyclotomic BMW algebra $\degbmw{n, S, r}$   (with parameters  $\Omega = (\omega_a)_{a \ge 0}$ and $u_1, \dots, u_r$) over a commutative ring $S$.
 Let 
 \begin{equation} \label{equation:  relate a's to u's}
 p(u) = (u - u_1) \cdots (u-u_r) = \sum_{j  =0}^r a_j u^j.
 \end{equation}
    The coefficients $a_j$  for $j < r$  are signed elementary symmetric functions in $u_1, \dots, u_r$,  
   namely  \break $a_j =  (-1)^{r-j} \varepsilon_{r-j}(u_1, \dots, u_r)$, 
    and $a_r = 1$.   
 \end{notation}
 
 \begin{lemma}  The left ideal $\both{2, S, r}\,  e_1$  in $\both{2, S, r}$  is equal to the $S$--span of  
 $\{e_1, y_1 e_1, \dots, y_1^{r-1} e_1 \}$.  
 \end{lemma} 
 
 \begin{proof}  For both the cyclotomic and degenerate cyclotomic BMW algebras, it is easy to check using the relations that the $S$--span of $\{e_1, y_1 e_1, \dots, y_1^{r-1} e_1 \}$ is invariant under multiplication by  the generators on the left.
 \end{proof}

 \begin{lemma} \label{lemma: periodicity of omegas when e1 is torsion free}
  Assume that $e_1$ is not a torsion element over $S$ in $\both{2, S, r}$.  Then the elements $\omega_j$,  $j  \ge 0$  satisfy the following recursion relation:
 \begin{equation} \label{equation:  periodicity condition on omega's}
  \sum_{j = 0}^r  a_j  \omega_{j + \ell}  = 0, \text{  for all } \ell  \ge 0.
  \end{equation}
 \end{lemma}
 
 \begin{proof}
 Multiply  the cyclotomic condition: $\sum_{j = 0}^r  a_j  y_1^j = 0$  by $y_1^\ell$, and then multiply from both sides by $e_1$.  Use the compression and idempotent relations to obtain    
 $\sum_{j = 0}^r  a_j  \omega_{j + \ell}  e_1 = 0$.  Since $e_1$ is not a torsion element over $S$,  the result follows. 
 \end{proof}
 
\begin{definition} Consider the cyclotomic or degenerate cyclotomic BMW algebras over a commutative ring $S$ with suitable parameters.  We say that the parameters are {\em admissible}
if $\{e_1, y_1 e_1, \dots,  y_1^{r-1} e_1 \}$ is linearly independent over $S$ in 
$\both{2, S, r}$.   
\end{definition} 

For both the cyclotomic and degenerate cyclotomic BMW algebras, admissibility as defined above translates into explicit conditions on the parameters.  We review this for the two classes of algebras separately in the following two sections.   
 
 \section{Admissibility for degenerate cyclotomic BMW algebras} 
 \label{section: admissibility in the degenerate case}
 Consider the degenerate cyclotomic BMW algebras $\degbmw{n, S,r}$
 with parameters $\Omega = (\omega_a)_{a \ge 0}$ and 
 $u_1, \dots, u_r$ over a  commutative ring $S$.  Define $a_0, \dots, a_{r-1}$  by (\ref{equation:  relate a's to u's}).

 \begin{lemma}[\cite{goodman-degenerate-admissibility}, Lemma 4.1]
  \label{lemma: linear independence implies admissibility}
Suppose that $\{ e_1, y_1 e_1, \dots, y_1^{r-1} e_1\}$ is linearly independent over $S$ in 
$\degbmw{2, S, r}$.  Then the parameters  satisfy the following relations:
\begin{equation}  \label{equation:  admissibility relation degenerate}
 \sum_{\mu = 0}^{r-j-1}  \omega_\mu  a_{\mu + j +1}  = - 2 \delta_{(\text{$r-j$ is odd})}\  a_j  + \delta_{(\text{$j$ is even})} \  a_{j+1},
\end{equation}
for $0 \le j \le r-1$. 
 \end{lemma}
 
 We are going to show that  admissibility (i.e.\  linear independence of $\{ e_1, y_1 e_1, \dots, y_1^{r-1} e_1\}$)  is equivalent to the parameters satisfying conditions (\ref{equation:  periodicity condition on omega's})  and (\ref{equation:  admissibility relation degenerate}).

 \begin{lemma}[\cite{goodman-degenerate-admissibility}, Lemma 4.4]  \label{corollary: universal polynomials degenerate admissibility} 
 There exist universal polynomials $H_a(\ubold_1, \dots, \ubold_r)$ for $a \ge 0$,  symmetric in $\ubold_1, \dots \ubold_r$,    with integer coefficients, such that whenever
$S$ is a commutative ring with   parameters \ $\Omega = (\omega_a)_{a \ge 0}$  and 
 $u_1, \dots, u_r$ satisfying (\ref{equation:  periodicity condition on omega's})  and (\ref{equation:  admissibility relation degenerate}), one has 
\begin{equation}\label{equation:  universal polynomials for degenerate admissibility}
\omega_a =  H_a(u_1, \dots, u_r) \quad  \text{  for  }  a \ge 0. 
\end{equation}
    Conversely, if $\omega_a =  H_a(u_1, \dots, u_r)$
  for $a \ge 0$,  then the parameters satisfy (\ref{equation:  periodicity condition on omega's})  and (\ref{equation:  admissibility relation degenerate}).
\end{lemma}

\begin{proof}   The system of relations (\ref{equation:  admissibility relation degenerate}) is a unitriangular linear system of equation for the variables $\omega_0,  \dots, \omega_{r-1}$.  In fact, if we list the equations in reverse order
then the matrix of coefficients is
$$
\begin{bmatrix}
1 & & & &\\
a_{r-1} & 1 & && \\
a_{r-2} & a_{r-1} & 1 & &\\
\vdots  & & \ddots& \ddots & \\
a_1   &a_2  &\dots&a_{r-1}& 1\\
\end{bmatrix}.
$$
Solving the system for $\omega_0,  \dots, \omega_{r-1}$  gives  these quantities as polynomial functions of 
$a_0,  \dots, a_{r-1}$,  thus symmetric polynomials in $u_1, \dots, u_r$.    The recursion relations
 $ \sum_{j = 0}^r  a_j  \omega_{j + m}  = 0$,  for all $m \ge 0$ yield (\ref{equation:  universal polynomials for degenerate admissibility})  for
 $a \ge r$.    The converse is obvious, since the $\omega_a$ given by (\ref{equation:  universal polynomials for degenerate admissibility}) are the solutions of the equations
  (\ref{equation:  periodicity condition on omega's})  and (\ref{equation:  admissibility relation degenerate}).
\end{proof}

 \subsection{The admissibility condition of Ariki, Mathas, and Rui}
 Ariki, Mathas and Rui used a different approach to admissibility for degenerate cyclotomic BMW algebras in their fundamental work~\cite{ariki-mathas-rui}.   Let $\ubold_1, \dots, \ubold_r$ and $t$  be algebraically independent indeterminants over $\Z$.   Define symmetric polynomials $q_a(\ubold)$  in $\ubold_1, \dots, \ubold_r$ by
$$\prod_{i=1}^r\frac{1+\ubold_it}{1-\ubold_it}=\sum_{a\ge0}q_a(\ubold)t^a.$$
The polynomials $q_a$  are known as {\em Schur  $q$--functions}.  Define
\begin{equation} \label{equation: definition of eta's}
\eta_{a}^{\pm}(\ubold) =  q_{a+1}(\ubold) \ \pm \  \frac{(-1)^{r-1}}{2} q_a(\ubold)  + \frac{1}{2} \delta_{a, 0}.
\end{equation}
for $a \ge 0$.    Then (cf.~\cite{ariki-mathas-rui}, Lemma 3.8)
\begin{equation} \label{equation: generating function of eta's}
\sum_{a \ge 0}  \eta_{a}^{\pm}(\ubold)    t^{-a} = (\frac{1}{2} - t)\  + \ 
              (t \pm \frac{(-1)^{r-1}}{2} )\prod_{i=1}^r\frac{t+u_i}{t-u_i},
\end{equation}
as one sees by expanding the series, using the definition of the Schur $q$--functions.
Ostensibly,  $\eta_a^{\pm}(\ubold)  \in \Z[1/2, \ubold_1, \dots, \ubold_r]$, but actually:

\begin{lemma}   \label{lemma:  why 1/2 is not needed} \mbox{}
\begin{enumerate}
\item \label{lemma:  why 1/2 is not needed item 1}
$q_0(\ubold) = 1$.
\item \label{lemma:  why 1/2 is not needed item 2}
For $a \ge 1$,  $q_a  \equiv 2 p_a(\ubold)  \mod  4 \Z[\ubold_1, \dots, \ubold_r]$, where
$p_a$ denotes the $a$--th power sum symmetric function.
\item \label{lemma:  why 1/2 is not needed item 3}
$\eta_a^{\pm}(\ubold)  \in \Z[\ubold_1, \dots, \ubold_r]$.
\end{enumerate}
\end{lemma}

\begin{proof}  
Part \refc{lemma:  why 1/2 is not needed item 1} is obvious. 
Using the identity:
$$
\frac{1 + vt}{1 - vt} = 1 + 2 vt(1 - vt)\inv = 1 + 2 vt + 2 v^2 t^2 + 2 v^3 t^3 + \dots,
$$
one sees that the coefficient of $t^a$ in  $\displaystyle \prod_{i=1}^r\frac{1+\ubold_it}{1-\ubold_it}$  
is $2 \sum_i  u_i^a$  plus  a sum of terms divisible by 4.  This gives  \refc{lemma:  why 1/2 is not needed item 2},  and \refc{lemma:  why 1/2 is not needed item 3} follows as well.  
\end{proof}  

\begin{example}  \label{example:  evaluation of eta's when characteristic is 2}
Consider a ring $S$ of characteristic 2 and $u_1, \dots, u_r \in S$.  Then
$q_a(u_1, \dots, u_r) = 0$  for $a \ge 1$,  but  $\frac{1}{2} q_a(u_1, \dots, u_r) = \sum_i u_i^a$; that is, 
we consider   $\frac{1}{2} q_a$ in $\Z[\ubold_1, \dots, \ubold_r]$, and then evaluate at $(u_1, \dots, u_r) \in S^r$.   Furthermore,   
$$\eta_0^+(u_1, \dots, u_r) = \delta_{(r \text{ is odd})} \quad \text{ and} \quad
\eta_a^+(u_1, \dots, u_r) =   p_a(u_1, \dots, u_r), $$
for $a \ge 1$.
\end{example}

\begin{definition}[\cite{ariki-mathas-rui}]  \label{definition: u admissibility for degenerate case}
Let $S$ be a commutative ring with parameters 
$\Omega = (\omega_a)_{a \ge 0}$
and $u_1, \dots, u_r$.    Say that the parameters are $(u_1, \dots, u_r)$--admissible, or that $\Omega$ is $(u_1, \dots, u_r)$--admissible,  if  for all $a \ge 0$,
\begin{equation} \label{equation: u admissibility relations}
 \omega_a = \eta_a^{+}(u_1, \dots, u_r).  
 \end{equation}
\end{definition}

\begin{lemma} {\rm (cf. \cite{goodman-degenerate-admissibility}, Lemma 5.1)}
\label{lemma:  equivalence of degenerate admissibility conditions 1}
\begin{enumerate}
\item  \label{lemma:  equivalence of degenerate admissibility conditions 1 item 1}
$\eta_a^+(\ubold_1, \dots, \ubold_r) =  H_a(\ubold_1, \dots, \ubold_r)$,  where $H_a$ are the polynomials in Lemma \ref{corollary: universal polynomials degenerate admissibility}.
\item  \label{lemma:  equivalence of degenerate admissibility conditions 1 item 2}
Let $S$ be a commutative ring with parameters 
$\Omega = (\omega_a)_{a \ge 0}$
and $u_1, \dots, u_r$.     The
parameters are  $(u_1, \dots, u_r)$--admissible if and only if they 
 satisfy (\ref{equation:  periodicity condition on omega's})  and (\ref{equation:  admissibility relation degenerate}).
\end{enumerate}
\end{lemma}

\begin{proof}  Part \refc{lemma:  equivalence of degenerate admissibility conditions 1 item 1} is proved in~\cite{goodman-degenerate-admissibility}, Section 5, by showing that
the sequence $(\eta_a^+(\ubold))_{a \ge 0}$  satisfies (\ref{equation:  periodicity condition on omega's})  and (\ref{equation:  admissibility relation degenerate}); that is, 
 \begin{equation} \label{equation:  periodicity condition on eta's}
  \sum_{j = 0}^r  \abold_j  \eta_{j + \ell}^+(\ubold)  = 0, \text{  for all } \ell  \ge 0, 
  \end{equation}
  and
  \begin{equation}  \label{equation:  admissibility relation degenerate for eta's}
 \sum_{\mu = 0}^{r-j-1}  \eta_\mu^+(\ubold) \  \abold_{\mu + j +1}  = - 2 \delta_{(\text{$r-j$ is odd})}\  \abold_j  + \delta_{(\text{$j$ is even})} \  \abold_{j+1},
\end{equation}
for $0 \le j \le r-1$,  where $\abold_j = (-1)^{r-j}  \varepsilon_{r-j}(\ubold_1, \dots, \ubold_r)$.
Part \refc{lemma:  equivalence of degenerate admissibility conditions 1 item 2}  follows from part \refc{lemma:  equivalence of degenerate admissibility conditions 1 item 1} together with Definition \ref{definition: u admissibility for degenerate case} and Lemma \ref{corollary: universal polynomials degenerate admissibility}.
\end{proof}

\subsection{Recovering the results of Ariki, Mathas, and Rui}
The main results of~\cite{ariki-mathas-rui} regarding degenerate cyclotomic BMW algebras are stated for ground rings $S$ in which $2$ is invertible.  
 The primary reason for this restriction on the ground ring was that it seemed to be needed in order to 
 use the quantities $\eta_a^{+}(u_1, \dots u_r)$, which play a central role in~\cite{ariki-mathas-rui}, via Definition \ref{definition: u admissibility for degenerate case}.
 Using Lemma \ref{lemma:  why 1/2 is not needed}, the restriction on the ground ring can be eliminated.  
We proceed to outline how the proofs have to be adjusted.

Let us define a universal  ring with $(u_1, \dots, u_r)$--admissible parameters.
Let $\ubold_1, \dots, \ubold_r$ be indeterminants over $\Z$.
Let $\mathcal Z = \Z[\ubold_1, \dots, \ubold_r]$;
define $\abold_j = (-1)^{j} \varepsilon_{r-j}(\ubold_1, \dots, \ubold_r)$  for $0 \le j \le r$, 
where $\varepsilon_k$ is the $k$--th elementary symmetric function, 
 and
define $\omegabold_a$ for $a \ge 0$ by 
\begin{equation}\label{equation:  universal polynomials for degenerate admissibility2}
\omegabold_a =  H_a(\ubold_1, \dots, \ubold_r) = \eta^+_a(\ubold_1, \dots, \ubold_r)  \text{  for  }  a \ge 0. 
\end{equation}
The parameters $\bm \Omega = (\omegabold_a)_{a \ge 0}$ and
$\ubold_1, \dots \ubold_r$  are  $(\ubold_1, \dots \ubold_r)$--admissible by definition.
(This is the same construction as in~\cite{ariki-mathas-rui}, page  105, except that we don't need to adjoin $1/2$ to the ring.) 
If $S$ is any
commutative ring with parameters $\Omega = (\omega_a)_{a \ge 0}$ and
$u_1, \dots,  u_r$, such that $\Omega$ is  $(u_1, \dots, u_r)$--admissible then there is a unique algebra homomorphism from $\mathcal Z$ to $S$ taking
$\ubold_j \mapsto u_j$. Since $\Omega$ is  $(u_1, \dots, u_r)$--admissible, it follows that $\omegabold_a \mapsto \omega_a$  for all $a \ge 0$.  For all $n \ge 0$,   we have
\begin{equation} \label{equation:  universal a admissible ring and specializations}
\degbmw{n, S, r}(\Omega; u_1, \dots, u_r)\ \cong \  \degbmw{n, \mathcal Z, r}(\bm \Omega; \ubold_1, \dots \ubold_r) \otimes_{\mathcal Z} S.
\end{equation}
 See~\cite{GH2}, Remark 3.4 for a justification.

 Let $S$ be any commutative ring with parameters $\Omega = (\omega_a)_{a \ge 0}$ and
$u_1, \dots,  u_r$ (with no conditions imposed on the parameters).   We recall a construction of a spanning set in $\degbmw{n} =  
\degbmw{n, S, r}(\Omega; u_1, \dots, u_r)$  from~\cite{ariki-mathas-rui}.  
 Remark that there is a homomorphism
 from the Brauer algebra $\mathcal B_n(\omega_0)$  with parameter 
 $\omega_0$  to $\degbmw{n, S, r}$  taking $s_i \mapsto s_i$  and $e_i \mapsto e_i$;  this follows from the presentation of the Brauer algebra cited in~\cite{ariki-mathas-rui},  Proposition 2.7.   
  For a Brauer diagram $\gamma$,  we will also write $\gamma$ for the image of $\gamma$ in $\degbmw{n, S, r} $.  The  ``$r$--regular monomials" in  $\degbmw{n, S, r}$ are defined to be  the elements 
  \begin{equation} \label{spanning elements of N_n}
 y^{\bm p}  \gamma  y^{\bm q},
 \end{equation}
 where $\gamma$ is a Brauer diagram,   $ y^{\bm p} = {y_1}^{p_1} \cdots {y_n}^{p_n}$,  and
 $ y^{\bm q} = {y_1}^{q_1} \cdots {y_n}^{q_n}$;  moreover,     $p_i$ and $q_i$ are non-negative integers,  in the interval $0, 1, \dots, r-1$,  and
 $p_i = 0$ unless the $i$-th vertex at the bottom of $\gamma$ is the left endpoint of a horizontal strand,  and 
 $q_i = 0$ unless the $i$-th vertex at the top of $\gamma$ is either the left endpoint of a horizontal strand, or the top endpoint of a vertical strand. Note that there are at most $n$ strictly positive exponents $p_i$ or $q_i$, and the number of $r$--regular monomials is $r^n (2n-1)! !$.

  \begin{proposition}[\cite{ariki-mathas-rui}, Proposition 2.15]  \label{proposition: spanning by r regular monomials}
   Let $S$ be any commutative ring with parameters $\Omega = (\omega_a)_{a \ge 0}$ and
$u_1, \dots,  u_r$.
  For all $n \ge 0$,  $\degbmw{n, S, r}(\Omega; u_1, \dots, u_r)$
  is spanned over $S$ by the set of $r$--regular monomials.   Furthermore, the ideal \ $\degbmw{n} e_{n-1} \degbmw{n}$ is spanned by those $r$--regular monomials  $y^{\bm p}  \gamma  y^{\bm q}$ such that $\gamma$ has at least two  horizontal strands.  
  \end{proposition}

 \begin{remark}  \label{remark:  prop. on spanning set does not need 2 invertible}
 It may appear from the presentation in ~\cite{ariki-mathas-rui} that this result depends on the invertibility of $2$ in the ground ring and on an additional condition on the parameters (called ``admissibility" in ~\cite{ariki-mathas-rui}, see Definition 2.10 in that paper).   However, in fact, the result does not depend on any additional assumptions.    From Theorem 2.12 in ~\cite{ariki-mathas-rui}, one only needs the statement that the degenerate affine BMW algebra is spanned by regular monomials, and the argument for this part of Theorem 2.12 is valid over an arbitrary ring.   The argument given for Proposition 2.15 itself in ~\cite{ariki-mathas-rui} is also valid over an arbitrary ring.
 \end{remark}

  \begin{theorem}[\cite{ariki-mathas-rui}]  \label{theorem:  generic semisimplicity for deg bmw from ARM}
  Let $F = \Q(\ubold_1, \dots, \ubold_r)$ denote the field of fractions of \ $\mathcal Z$.  Then the algebra $\degbmw{n, F, r}(\bm \Omega; \ubold_1, \dots, \ubold_r)$ is split--semisimple of dimension $r^n (2n-1)! !$. 
  \end{theorem} 

This theorem is proved by explicit construction of sufficiently many irreducible representations.

\begin{corollary} {\rm (cf.~\cite{ariki-mathas-rui}, Theorem 5.5)}  \label{corollary:  freeness of deg BMW with admissible parameters}
Let $S$ be a commutative ring with  parameters $\Omega = (\omega_a)_{a \ge 0}$   and 
$u_1, \dots, u_r$.  Assume that $\Omega$ is $(u_1, \dots, u_r)$--admissible. Then for all $n \ge 0$,  $\degbmw{n, S, r}(\Omega; u_1, \dots, u_r)$ is a free
$S$--module with basis the set of $r$--regular monomials.
\end{corollary}

\begin{proof}  Because of (\ref{equation:  universal a admissible ring and specializations}), it suffices to show that $\degbmw{n, \mathcal Z, r}(\bm \Omega; \ubold_1, \dots, \ubold_r)$ is a free
$\mathcal Z$--module with basis the set $\mathcal M$ of $r$--regular monomials.   By Proposition
\ref{proposition: spanning by r regular monomials},  $\mathcal M$ is a spanning set, and
$\mathcal M \otimes 1 := \{m \otimes 1 : m \in \mathcal M\}$ is a spanning set  in $\degbmw{n, \mathcal Z, r} \otimes_{\mathcal Z} F =
\degbmw{n, F, r}$.  But by Theorem \ref{theorem:  generic semisimplicity for deg bmw from ARM}, the latter algebra over $F$ has dimension $r^n (2n-1)! !$, and hence $\mathcal M \otimes 1$ is linearly independent over $F$.  It follows that $\mathcal M$ is linearly independent over $\mathcal Z$.  
\end{proof}

The following theorem concerns cellularity of degenerate cyclotomic BMW algebras.  The definition of cellularity is given in Section \ref{section: cellularity}.

\begin{theorem}  {\rm (cf.~\cite{ariki-mathas-rui}, Theorem 7.17)}
 Let $S$ be a commutative ring with  parameters $\Omega = (\omega_a)_{a \ge 0}$   and 
$u_1, \dots, u_r$.  Assume that $\Omega$ is $(u_1, \dots, u_r)$--admissible.     Then $\degbmw{n, S, r}(\Omega; u_1, \dots, u_r)$ is a cellular algebra of rank
$r^n (2n -1)!!$.  
\end{theorem}

\begin{proof}  Because of (\ref{equation:  universal a admissible ring and specializations}), it suffice to prove this when $S = \mathcal Z$.    For this special case, one can follow the proof in 
~\cite{ariki-mathas-rui}, Theorem 7.17,  substituting Corollary \ref{corollary:  freeness of deg BMW with admissible parameters} for~\cite{ariki-mathas-rui}, Theorem 5.5.    For an alternative treatment of cellularity, see~\cite{goodman-graber2}, Section 6.5. 
\end{proof}

 \subsection{Equivalence of admissibility conditions}
 The following theorem establishes the equivalence of the various admissibility criteria  for degenerate cyclotomic BMW algebras.
 
 \begin{theorem} {\rm (cf. \cite{goodman-degenerate-admissibility}, Theorem 5.2)}   \label{theorem:  equivalent admissibility conditions for degenerate algebras}
 Let $S$ be a commutative ring with  parameters  $\Omega = (\omega_a)_{a \ge 0}$   and 
$u_1, \dots, u_r$.     
Consider the degenerate cyclotomic BMW algebra $\degbmw 2 = \degbmw{2, S, r}(\Omega;  u_1, \dots, u_r)$.
The following are equivalent:
\begin{enumerate}
\item  \label{theorem:  equivalent admissibility conditions for degenerate algebras item 1}
The parameters are admissible, i.e.\  
$\{ e_1, y_1 e_1, \dots, y_1^{r-1} e_1\}$ is linearly independent over $S$ in $\degbmw 2$.
\item  \label{theorem:  equivalent admissibility conditions for degenerate algebras item 2}
$\{ y_1^a e_1 y_1^b,  s_1 y_1^a y_2^b, y_1^a y_2^b  : 0 \le a, b \le r-1\} $ is an $S$--basis of \   $\degbmw 2$.  
\item  \label{theorem:  equivalent admissibility conditions for degenerate algebras item 3}
The parameters satisfy (\ref{equation:  periodicity condition on omega's})  and (\ref{equation:  admissibility relation degenerate}).
\item \label{theorem:  equivalent admissibility conditions for degenerate algebras item 4}
The parameters are  $(u_1, \dots, u_r)$--admissible.  \end{enumerate}
\end{theorem}

\begin{proof}    $\refc{theorem:  equivalent admissibility conditions for degenerate algebras item 1} \implies \refc{theorem:  equivalent admissibility conditions for degenerate algebras item 3}$    results from Lemmas \ref {lemma: periodicity of omegas when e1 is torsion free} and \ref {lemma: linear independence implies admissibility}. 
We have 
 $\refc{theorem:  equivalent admissibility conditions for degenerate algebras item 3} \iff \refc{theorem:  equivalent admissibility conditions for degenerate algebras item 4}$   by Lemma \ref{lemma:  equivalence of degenerate admissibility conditions 1}. 
The implication   $\refc{theorem:  equivalent admissibility conditions for degenerate algebras item 4} \implies \refc{theorem:  equivalent admissibility conditions for degenerate algebras item 2}$  is a special case of Corollary \ref{corollary:  freeness of deg BMW with admissible parameters}.  Finally, the implication   $\refc{theorem:  equivalent admissibility conditions for degenerate algebras item 2} \implies \refc{theorem:  equivalent admissibility conditions for degenerate algebras item 1}$  is trivial. 
\end{proof}

  \section{Admissibility for  cyclotomic BMW algebras}  \label{subsection: admissibility in the non degenerate case}
 Fix an integral domain  $S$ with parameters $\rho$, $q$,  $\Omega = (\omega_a)_{a \ge 0}$  and 
 $u_1, \dots, u_r$;  assume that $\rho$ and $q$ are invertible and that  equation (\ref{equation:  basic relation in ground ring})  holds.  Consider the  cyclotomic BMW algebras $\bmw{n, S,r} = 
 \bmw{n, S, r}(\rho, q, \Omega;  u_1, \dots, u_r)$.

\subsection{Admissibility conditions of Wilcox and Yu}

Explicit relations on the parameters that are equivalent to admissibility (i.e.\ linear independence of $\{ e_1, y_1 e_1,  \dots, y_1^{r-1} e_1\}$)  have been found by Wilcox and Yu~\cite{Wilcox-Yu,Wilcox-Yu3,Yu-thesis}.   The form of these relations depends on whether  $q^2 \ne 1$  is satisfied in $S$.  
Note that the conditions $q^2 \ne 1$  (in the non--degenerate case)  and  $\char(S) \ne 2$  (in the degenerate case)  should be regarded as analogous.  

\begin{theorem}[Wilcox \& Yu, \cite{Wilcox-Yu}] \label{theorem: equivalent conditions for admissibility}
 Let $S$ be an integral domain  with
parameters $\rho$, $q$, $\Omega = (\omega_a)_{a \ge 0}$,   and $u_1, \dots, u_r$ satisfying Equation (\ref{equation:  basic relation in ground ring})   and   
 $(q - q\inv) \ne 0$.
The following conditions are equivalent:
\begin{enumerate}
\item  $\{e_1, y_1 e_1, \dots, y_1^{r-1} e_1\} \subseteq \bmw{2, S, r}$ is linearly independent over $S$.
\item  The parameters satisfy the recursion relation (\ref{equation:  periodicity condition on omega's})
and the following relations:

\begin{equation} \label{equation: yu wilcox admissibility condition 1}
\begin{aligned}
 (q-q\inv)  \bigg [ \sum_{j = 1}^{r - \ell} a_{j+\ell} \omega_j 
\bigg ]  =  
 -\rho(& a_\ell -  a_{r-\ell}/a_0) \   \\
& +  \ (q-q\inv) \bigg [ \sum_{j = \max(\ell + 1, \lceil r/2 \rceil)}^{\lfloor (\ell + r)/2 \rfloor} a_{2j - \ell}
-   \sum_{j =  \lceil \ell/2 \rceil}^{\min(\ell, \lceil r/2 \rceil -1)} a_{2j - \ell} \bigg ] \\
 \end{aligned}
\end{equation}
{for $1 \le \ell \le r-1$},  and 
\begin{equation} \label{equation: yu wilcox admissibility condition 2}
\text{ $\rho = \pm a_0$ if $r$ is odd, and $\rho \in \{q\inv a_0,  -q a_0\}$ if $r$ is even.}
\end{equation}
\end{enumerate}
\end{theorem}

Note that Equations (\ref{equation:  basic relation in ground ring}), (\ref{equation: yu wilcox admissibility condition 1}),  and 
 (\ref{equation: yu wilcox admissibility condition 2})  determine   $\omega_0, \dots, \omega_{r-1}$ and
 $\rho$   in terms of $q$, $u_1, \dots, u_r$  while
 the recursion relation (\ref{equation:  periodicity condition on omega's})
determines   $\omega_a$  for $a \ge r$. 

In~\cite{Wilcox-Yu3}  and~\cite{Yu-thesis}  Wilcox and Yu derive  explicit relations on the parameters that are equivalent to linear independence of $\{e_1, y_1 e_1, \dots, y_1^{r-1} e_1\}$
also in the case that $q - q\inv = 0$;  their new conditions reduce to   those of Theorem 
\ref {theorem: equivalent conditions for admissibility} in the case $q - q\inv \ne 0$.

\subsection{The admissibility criterion of Rui and Xu}  \label{subsection: admissibility criterion of Rui and Xu}
Rui and Xu~\cite{rui-2008},  following~\cite{ariki-mathas-rui},    take a different approach to 
admissibility for cyclotomic BMW algebras when $q - q\inv \ne 0$.   
Let $\ubold_1,  \dots, \ubold_r$,   $\rhobold$,  $\qbold$, and   $t$  be algebraically independent indeterminants over $\Z$.     
Define
\begin{equation}\label{equation: definition of G(t)}
G(t) = G(\ubold_1, \dots, \ubold_r; t) = \prod_{\ell = 1}^r  \frac{t - \ubold_\ell}{t \ubold_\ell -1}.
\end{equation}
Let
\begin{equation} \label{equation: formula for Z of t}
Z(t) = Z(t; \ubold_1, \dots, \ubold_r, \rhobold, \qbold) = -\rhobold\inv   +  (\qbold -\qbold\inv)\frac{t^2}{t^2 -1} + A(t) \  G(t\inv),
\end{equation}
where
\begin{equation} \label{equation: formula for Z of t 2}
A(t) =
\begin{cases}
{ \textstyle \rho\inv   (\prod_j \ubold_j )}\ +\  (\qbold -\qbold\inv){t}/{(t^2-1)}  \quad &\text{if $r$ is  odd, and } \\
{ \textstyle \rho\inv   (\prod_j \ubold_j )} \ -\  (\qbold -\qbold\inv){t^2}/{(t^2-1)} \quad &\text{if $r$ is  even}.
\end{cases}
\end{equation}

\begin{definition}[Rui and Xu, \cite{rui-2008}]  \label{definition: RX admissibility}
 Let $S$ be an integral domain with
parameters $\rho$, $q$,  $\Omega = (\omega_a)_{a \ge 0}$   and $u_1, \dots, u_r$ satisfying (\ref{equation:  basic relation in ground ring}) and  
$q - q\inv \ne 0$. One says that 
the parameters are {\em $(u_1, \dots, u_r)$--admissible},   or that $\Omega$ is  {\em $(u_1, \dots, u_r)$--admissible}, if  
\begin{equation} \label{equation: defn of u-admissibility in non degenerate bmw case}
(q - q\inv) \sum_{a \ge 0}  \omega_a t^{-a} =    Z(t; u_1, \dots, u_r, \rho, q),
\end{equation} 
where $Z$ is defined in equations (\ref{equation: formula for Z of t}) and (\ref{equation: formula for Z of t 2}).
\end{definition}

\begin{remark} \label{remark:  conditions on rho, assuming u-admissibility}
 Let  $S$ be an integral domain with $(u_1, \dots, u_r)$--admissible parameters, as in Definition 
\ref{definition: RX admissibility}.  With $p = \prod_j u_j$,       we have   
\begin{equation}
\label{equation: conditions on rho in u-admissibility}
\rho = \pm p  \text{\ \ if  $r$  is odd, and  }  \rho \in \{q\inv  p,  -q  p\} \text{\ \ if $r$ is even.}
\end{equation}
The condition 
(\ref{equation: conditions on rho in u-admissibility})
on $\rho$ was included in the definition of $u$--admissibility in 
\cite{rui-2008}, but it actually follows from (\ref{equation:  basic relation in ground ring}) and (\ref{equation: defn of u-admissibility in non degenerate bmw case}), as explained in
~\cite{goodman-admissibility}, Remark 3.10. 
\end{remark}

\subsection{Equivalence of admissibility conditions}
The following theorem establishes the equivalence of the various admissibility conditions for cyclotomic BMW algebras,  in case the ground ring is integral and $q - q\inv \ne 0$.  
 
 \begin{theorem}[\cite{goodman-admissibility}, Theorem 4.4]    \label{theorem: equivalence of admissibility and u-admissibility in the non-degenerate case}
 Let $S$ be an integral domain  with
parameters $\rho$, $q$, $\Omega = (\omega_a)_{a \ge 0}$, and $u_1, \dots, u_r$ satisfying Equation (\ref{equation:  basic relation in ground ring})   and   
 $(q - q\inv) \ne 0$.
The following  are equivalent:
\begin{enumerate}
\item $\{e_1, y_1 e_1, \dots, y_1^{r-1} e_1\} \subseteq \bmw{2, S, r}$ is linearly independent over $S$.
\item  The parameters satisfy the recursion relation (\ref{equation:  periodicity condition on omega's})  and the conditions (\ref{equation: yu wilcox admissibility condition 1})  and 
 (\ref{equation: yu wilcox admissibility condition 2})  of Wilcox and Yu.
\item $\Omega$ is   $(u_1, \dots, u_r)$--admissible.
\end{enumerate}
 \end{theorem}

 \section{Semi--admissibility} \label{section: semi admissibility}
 Let $\both{n, S, r} = \both{n, S, r}(u_1, \dots, u_r)$  denote either the cyclotomic BMW algebra $\bmw{n,S, r}$,  
 with parameters $\rho$,  $q$,   $\Omega = (\omega_a)_{a \ge 0}$  and 
 $u_1, \dots, u_r$,  
 or the degenerate cyclotomic BMW algebra $\degbmw{n, S, r}$,
  with parameters   $\Omega = (\omega_a)_{a \ge 0}$  and  $u_1, \dots, u_r$,  
over an integral domain $S$.

From here on, we impose the following standing assumption:  
 
 \medskip
\begin{assumption} \label{assumption: torsion free}  The
ground ring $S$  is an integral domain, and the left ideal  $\mathcal A_{2, S, r}\, e_1 \subseteq \mathcal A_{2, S, r}$    is torsion free as an $S$--module. 
 \end{assumption}

 This assumption holds, in particular, whenever $S$ is a field. 

 \medskip
Under  Assumption \ref{assumption: torsion free}, exactly   
 one of the following three possibilities must hold:  
 \begin{enumerate}
 \item  $e_1 = 0$ in $\both{2, S, r}$.   In this case,   $e_{n-1} = 0$ in 
 $\both{n, S, r}$  for all $n\ge 2$.  
 The (degenerate)  cyclotomic BMW algebras reduce to (degenerate) cyclotomic Hecke algebras.
 \item   The parameters are admissible, i.e.\
 $\{e_1, y_1 e_1, \dots, y_1^{r-1} e_1\}$ is linearly independent over $S$ in $\mathcal A_{2, S, r}$.  This case has been studied in the literature and is  well understood.
 \item   There is a $d$ with    $0 < d < r$   such that $\{e_1, y_1 e_1, \dots, y_1^{d-1} e_1\}$ is linearly independent over $S$ in $\mathcal A_{2, S, r}$,  but $\{e_1, y_1 e_1, \dots, y_1^{d} e_1\}$ is linearly dependent.    This case remains to be investigated.
 \end{enumerate}
 
\begin{definition} \label{definition: semi admissible}
 Consider the cyclotomic or degenerate cyclotomic BMW algebras $\both{n, S, r}$ over an integral domain $S$ with suitable parameters.  Let $0 < d  < r$.    We say that the parameters are {\em $d$--semi--admissible}  if
$\{e_1, y_1 e_1, \dots, y_1^{d-1} e_1\}$ is linearly independent over $S$ in $\mathcal A_{2, S,r}$,  but $\{e_1, y_1 e_1, \dots,  y_1^{d} e_1\}$ is linearly dependent. 
\end{definition}

Suppose $d$--semi--admissibility of the parameters.    Then there is a polynomial  of $p_0(u) \in S[u]$ of degree $d$
 such that  $p_0(y_1) e_1 = 0$ but $r(y_1) e_1 \ne 0$  for any non--zero polynomial $r(u) \in S[u]$  of degree less than $d$.      Let $F$ denote the field of fractions of $S$, and write 
 $p(u) = (u - u_1)\cdots (u- u_r) \in S[u]$.     Since $p(y_1) = 0$, it follows that
 $p_0(u)$  divides $p(u)$  in $F[u]$.    Because of unique factorization in
 $F[u]$,  we have (after renumbering the roots $u_i$ of $p(u)$)   $p_0(u)  = \alpha (u-u_1) \cdots (u - u_d)$ for some non--zero $\alpha$ in $F$.   In fact $\alpha \in S$,  since it is the leading coefficient of $p_0(u)$.    Then we have 
 $\alpha (y_1 - u_1) \cdots (y_1 - u_d) e_1 = 0$.   Because we assumed $\mathcal A_{2, S, r} e_1$ is torsion--free over $S$, we can conclude that  $ (y_1 - u_1) \cdots (y_1 - u_d) e_1 = 0$.    Thus without loss of generality,  we can take $p_0(u) = (u - u_1) \cdots (u- u_d)$.  
 
 \begin{assumption}
 For the remainder of this section, we assume the parameters of $\both{n, S, r}$ are $d$--semi--admissible for some $d$ with   $0 < d < r$.    Assume without loss of generality that $p_0(y_1) e_1 = 0$,  where
 $
 p_0(u) = (u - u_1) \cdots (u- u_d) = \sum_{j = 0}^d  b_j u^j.  
$
\end{assumption}

\begin{lemma}  \label{lemma: homomorphism from A n r to A n d}
There is a surjective homomorphism $\theta : 
  \both{n, S, r}(u_1, \dots, u_r) \to \both{n, S, d}(u_1, \dots, u_d)$  taking generators to generators.   Moreover, $\theta$ maps the ideal generated by $e_{n-1}$ in   $\both{n, S, r}(u_1, \dots, u_r)$ onto the ideal generated by $e_{n-1}$ in  $\both{n, S, d}(u_1, \dots, u_d)$. 
\end{lemma}

\begin{proof}  The existence of the surjective homomorphism $\theta$ is evident because  the generators of \break $\both{n, S,  d}(u_1, \dots, u_d)$  satisfy the defining relations of 
  $\both{n, S, r}(u_1, \dots, u_r)$.   
  
 In general, if $A$ and $B$ are algebras and $\varphi: A \to B$ is a surjective algebra homomorphism, then for any $e \in A$, we have $\varphi(AeA) = B \varphi(e) B$.  In particular, $\theta$ maps the ideal generated  by 
 $e_{n-1}$ in $\both{n, S r}(u_1, \dots, u_r)$ onto the ideal generated by $e_{n-1}$ in $\both{n, S, d}(u_1, \dots, u_e)$.  
\end{proof}

 \begin{lemma}  \label{lemma:   semi--admissibility implies admissibility in fewer variables}   \mbox{}
  \begin{enumerate}
 \item   \label{lemma:   semi--admissibility implies admissibility in fewer variables item 1}
 The sequence \ $\Omega = (\omega_a)_{a \ge 0}$   satisfies the recurrence relation
 $\sum_{j = 0}^d  b_j  \omega_{j + \ell} = 0  \text{ for all }  \ell \ge 0.$ 
 \item  \label{lemma:   semi--admissibility implies admissibility in fewer variables item 2}
 The parameters  
 $\Omega = (\omega_a)_{a \ge 0} \text{ and } u_1, \dots u_d$
   in the degenerate case 
 (respectively, 
  $\rho, q, \  \Omega = (\omega_a)_{a \ge 0}$,   and  $u_1, \dots u_d$
in the non-degenerate case)   are admissible.  That is,  the set \break $\{e_1, y_1 e_1, \dots,  y_1^{d-1} e_1\}$  is linearly independent
 over $S$  in $\both{2, S, d}(u_1, \dots, u_d)$.   
 \end{enumerate}
 \end{lemma}

 \begin{proof}   For part \refc{lemma:   semi--admissibility implies admissibility in fewer variables item 1},    multiply the equation $p_0(y_1) e_1 = 0$  by   $e_1  y_1^\ell$  on the left, and employ the idempotent and compression relations to get $\sum_{j = 0}^d  b_j  \omega_{j + \ell}\, e_1 = 0  $.  The conclusion follows since $e_1$ is not a torsion element over $S$.  
 
 We should pause to see why something needs to be proved for part \refc{lemma:   semi--admissibility implies admissibility in fewer variables item 2}.   We have assumed that \break $\{e_1, y_1 e_1, \dots, y_1^{d-1} e_1\} \subseteq \both{2, S, r}(u_1, \dots, u_r)$  is linearly independent, and we have to prove that  \break $\{e_1, y_1 e_1, \dots, y_1^{d-1} e_1\} \subseteq \both{2, S, d}(u_1, \dots, u_d)$ is linearly independent.    The latter set is the image of the former under the algebra homomorphism
 $\theta : 
  \both{2, S, r}(u_1, \dots, u_r) \to \both{2, S, d}(u_1, \dots, u_d).$

Consider the degenerate case.  Apply the proof of 
 $\refc{theorem:  equivalent admissibility conditions for degenerate algebras item 1} \implies \refc{theorem:  equivalent admissibility conditions for degenerate algebras item 3}$  
 in Theorem \ref{theorem:  equivalent admissibility conditions for degenerate algebras} 
to  the linearly independent set  
$\{e_1, y_1 e_1, \dots, y_1^{d-1} e_1\} \subseteq \degbmw{2, S, r}(u_1, \dots, u_r)$.
This yields the analogue of condition (\ref{equation:  admissibility relation degenerate}) with $r$ replaced by $d$ and $a_j$ by $b_j$, namely
\begin{equation}  \label{equation:  d--admissibility relation degenerate}
 \sum_{\mu = 0}^{d-j-1}  \omega_\mu  b_{\mu + j +1}  = - 2 \delta_{(\text{$d-j$ is odd})}\  b_j  + \delta_{(\text{$j$ is even})} \  b_{j+1},
\end{equation}
for $0 \le j \le d-1$.  Part 
 \refc{lemma:   semi--admissibility implies admissibility in fewer variables item 1} 
   of this lemma together with the implication
    $\refc{theorem:  equivalent admissibility conditions for degenerate algebras item 3} \implies \refc{theorem:  equivalent admissibility conditions for degenerate algebras item 1}$  
  in Theorem \ref{theorem:  equivalent admissibility conditions for degenerate algebras},  applied now to $\degbmw{2, S, d}(\Omega; u_1, \dots, u_d)$,  gives the conclusion  
  \eqref{lemma:   semi--admissibility implies admissibility in fewer variables item 2}.

  For the non-degenerate case, one uses the theorem of Wilcox and Yu  
  (Theorem \ref{theorem: equivalent conditions for admissibility} in the case $q - q\inv \ne 0$,   or~\cite{Wilcox-Yu3} in general)   in the same manner. 
 \end{proof}

 \subsection{A spanning set for ${ \bmw{n, S, r}}$}
 
 In this section,  write $\bmw{n}$  for $\bmw{n, S, r}(\rho, q, \Omega;  u_1, \dots, u_r)$.

 Define elements  $y'_j$  and $y''_j$ for $j \ge 1$ in the affine or cyclotomic BMW algebras
 by  
$$ \begin{aligned}
y'_1 &= y''_1 = y_1, \\
y'_j  &= g_{j-1} y'_{j-1} g_{j-1}\inv \quad   \text{ and } \quad
y_j'' = g_{j-1}\inv  y''_{j-1}  g_{j-1} \text{ for }  j \ge 2.
\end{aligned}
$$
Since the elements $y'_j$ and $y''_j$  are all conjugate, we have
$
p(y'_j) = (y'_j - u_1) \cdots (y'_j - u_r) = 0,
$ 
for all $j$,  and similarly for the elements $y''_j$.  

\begin{lemma}  \label{lemma:  commutation of y', y''}
In any affine or cyclotomic BMW algebra,
$e_i$  and $g_i$  commute with $y'_j$  and $y''_j$   if $j \not\in \{i, i+1\}$. 
\end{lemma}

\begin{proof}  We will prove the commutation relations for the element $y'_j$; the proof for the elements $y''_j$ is essentially the same.

For $i \ge 2$,  $e_i$ and $g_i$  commute with $y_1$ and with
$g_1\ppm, \dots, g_{i-2}\ppm$, hence with $y'_j$  for $j < i$.       
One sees from the defining relations that 
\begin{equation} \label{equation: shift by two}
g_i g_{i+1} e_i g_{i+1}\inv g_i\inv = e_{i+1} \quad
\text{and} \quad
{g_i }\inv {g_{i+1}}\inv  e_i g_{i+1} g_i = e_{i+1}.
\end{equation}   
for all $i$.   
Using this,  and the already established commutation relation $[e_{i+1}, y'_i] = 0$, we have
\begin{equation}
\begin{aligned}
e_i y'_{i+2} &= e_i (g_{i+1} g_i)\, y'_i\, (g_i\inv g_{i+1}\inv) \\
&=  (g_{i+1} g_i)e_{i+1}\, y'_i\, (g_i\inv g_{i+1}\inv)  \\
&=  (g_{i+1} g_i)\, y'_i\, e_{i+1} (g_i\inv g_{i+1}\inv) \\
&=  (g_{i+1} g_i)\, y'_i\,  (g_i\inv g_{i+1}\inv) \,e_i = y'_{i+2}\,  e_i. 
\end{aligned}
\end{equation}
Similarly, using the braid relations and the commutation relation 
$[g_{i+1}, y'_i]  = 0$, we  obtain that $[g_i, y'_{i+2}]  = 0$.
 If $j \ge  i+3$, we have
 \begin{equation} \label{equation for commutation of y', y''}
y'_j =  (g_{j-1} \cdots g_{i+2}) y'_{i+2} (g_{i+2}\inv \cdots g_{j-1}\inv),
 \end{equation}
 and  we see that $g_i$ and $e_i$ commute with $y'_j$ because they commute with all the factors on the right hand side of (\ref{equation for commutation of y', y''}).
\end{proof}

\begin{lemma} \label{lemma:  reduction lemma for powers of y', y''}
In $\bmw{n}$, we have
  $p_0(y'_j)\, e_i = e_i  \,p_0(y'_j) = 0$  for all $j  \ne  i+ 1$.  
The same statement holds with $y'_j$  replaced by $y''_j$.  
\end{lemma}

\begin{proof}  We will verify explicitly that $p_0(y'_j) e_i= 0$ for $j  \ne  i+ 1$.  An identical argument shows the same with $y_j'$ replaced by $y_j''$, and the statement
$e_i p_0(y_j') = e_i p_0(y_j'') = 0$ for $j \ne i+1$ follows as well by applying the involution $*$.  

We first show  that $p_0(y'_j) e_j = 0$ for all $j$, by induction.   This is given for $j = 1$.       If 
 $p_0(y'_j) e_j = 0$  holds for some particular value of $j$,  then
$$
0 = g_j g_{j+1} p_0(y'_j) e_j g_{j+1}\inv g_j\inv  =  p_0(y'_{j+1}) e_{j+1},
$$
and our assertion follows.

Next we check that  $p_0(y'_j) e_i= 0$  for all $j \le i$, by induction on $i - j$.   We have already checked the case $j = i$.    If this holds for some particular $j \le i$,  then
$$
p_0(y'_j)  e_{i+1} = p_0(y'_j)  e_{i+1} e_{i} e_{i+1}  =  e_{i+1} p_0(y'_j)  e_{i} e_{i+1} = 0.   
$$
It remains to check that $p_0(y'_j) e_i= 0$ for  $i \le j-2$.    We have
$$
\begin{aligned}
p_0(y'_j) \, e_i &=   g_{j-1}\cdots     (g_{i+1} g_i) p_0(y'_i ) ({g_i}\inv {g_{i+1}}\inv) \cdots {g_{j-1}}\inv \, e_i \\
&=  g_{j-1}\cdots     (g_{i+1} g_i)  [p_0(y'_i ) e_{i+1}  ]  ({g_i}\inv {g_{i+1}}\inv) \cdots {g_{j-1}}\inv= 0,
\end{aligned}
$$
since $p_0(y'_i ) e_{i+1} = 0$ by the previous part of the proof.  
\end{proof}

We now describe a certain basis of the affine BMW algebra  $\abmw{n} = \abmw{n, S}(\rho, q, \Omega)$  that was introduced in 
~\cite{goodman-2008},  Section 3.2.    Given a permutation $\pi \in \S_n$,   with
reduced expression $\pi = s_{i_1} s_{i_2} \cdots s_{i_\ell}$,  let
$g_\pi =  g_{i_1} g_{i_2} \cdots g_{i_\ell}$  in $\abmw{n}$;  in fact,  $g_\pi$ is independent of the choice of the reduced expression of $\pi$,  see~\cite{goodman-2008}, Section 2.4.  Fix an integer $f$ with
$0 \le 2f \le n$, and let 
$\gamma$ be a Brauer diagram with $2f$ horizontal strands and $s = n-2f$ vertical strands.     Then $\gamma$ has a unique factorization in the Brauer algebra of the form
\begin{equation}\label{equation:  factorization of Brauer diagram}
\gamma =   \alpha (e_1 e_3 \cdots e_{2f-1})  \pi   \beta\inv,
\end{equation}
where  $\pi$ is a permutation of $\{2f+1, \dots, n-1, n\}$   and $\alpha$ and $\beta$ are in a certain subset  
 $\dfn$ of $\S_n$ described in~\cite{goodman-2008},  Section 3.2.   Consider a sequence of $n$ integers
 $$(\bm a, \bm b, \bm c) = (a_1, a_3, \dots, a_{2f-1},  b_1, b_3, \dots, b_{2f-1},  c_{2f + 1},  \dots, c_n).$$  
Corresponding to $\gamma$ and the sequence $(\bm a, \bm b, \bm c)$,  we let $T_{\gamma, \bm a, \bm b, \bm c}$
be the following element of  $\abmw{n} $, 
$$
T_{\gamma, \bm a, \bm b, \bm c} =  g_\alpha\, {y''}^{\,\bm a}\,  (e_1 e_3 \cdots e_{2f-1})  g_\pi   {y''}^{\,\bm c}  
 {y'}^{\,\bm b} (g_\beta)^*,
$$
where  
$$ {y''}^{\,\bm a} = {y''_1}^{\,a_1}{y''_3}^{\,a_3}\cdots {y''_{2f-1}}^{a_{2f-1}},$$
 $$ {y'}^{\,\bm b} =  {y'_{2f-1}}^{b_{2f-1}}\cdots {y'_3}^{\,b_3}  {y'_1}^{\,b_1},$$
 and
 $$
 {y''}^{\,\bm c} = {y''_{n}}^{c_{n}}      \cdots      {y''_{2f+2}}^{c_{2f + 2}}     {y''_{2f+1}}^{c_{2f + 1}}
 $$
 
 If $\gamma$ has no horizontal strands (i.e.\ $\gamma$ is a permutation diagram), the elements  $T_{\gamma, \bm a, \bm b, \bm c} $  still make sense, but then $f = 0$,  $\alpha$ and $\beta$ are trivial,  $\gamma = \pi$, and $\bm a$ and $\bm b$ are empty sequences.   We have
 $$
 T_{\gamma, \bm a, \bm b, \bm c}  =  T_{\gamma,  \bm c}  =  g_\gamma  y''^{\,\bm c}.  
 $$

It is shown in~\cite{goodman-2008},  Section 3.2 that the set of $T_{\gamma, \bm a, \bm b, \bm c}$, as $\gamma$ ranges over Brauer diagrams and $(\bm a, \bm b, \bm c)$ ranges over $n$--tuples of integers forms an $S$--basis of $\abmw{n}$, and, moreover,  the subset corresponding to Brauer diagrams with $2 f >0$ horizontal strands, forms  a basis of the ideal $\abmw{n} e_{n-1}  \abmw{n}$.  
 
 Let  $b'(n)$ denote the number of Brauer diagrams on $n$ strands with at least one horizontal strand, 
 $b'(n) =  (2n - 1) !!  -  n!$.   
 
 \begin{lemma}  \label{lemma:  spanning set for the ideal in cyclotomic algebra}
 The ideal $\bmw{n} e_{n-1} \bmw{n}$    is  spanned by a set of  $d^n b'(n)$ elements.   The algebra  $\bmw{n}$  is spanned by a set of 
  $d^n b'(n) + r^n n!$ elements. 
 \end{lemma}
 
 \begin{proof}  We also write  $T_{\gamma, \bm a, \bm b, \bm c}$  for the image of that element in 
the cyclotomic BMW algebra  $\bmw{n}$.       The set of all $T_{\gamma, \bm a, \bm b, \bm c}$ spans  $\bmw{n}$,  while those with
 $\gamma$ a Brauer diagram with $2f > 0$ horizontal strands span the ideal $\bmw{n} e_{n-1} \bmw{n}$.   
 
If $\gamma$ is a permutation diagram,  then we can write any element $
 T_{\gamma, \bm a, \bm b, \bm c}  = T_{\gamma, \bm c}$  as a linear combination of elements
 $T_{\gamma, \bm c'}$,  with $0 \le c'_i \le r$,  using the relations
 $
p(y''_j) = (y''_j - u_1) \cdots (y''_j - u_r) = 0.
$

In the following,  take $f >0 $ and let  $\gamma$  be a Brauer diagram with $2f$ horizontal strands.   
We claim that any element  $T_{\gamma, \bm a, \bm b, \bm c}$ can be written as a linear combination of
 elements  $T_{\gamma, \bm a', \bm b', \bm c'}$  where  $a'_i$,   $b'_i$,  and $c'_i$  lie in the interval $0, 1,  \dots, d-1$. 
 Using the commutation relations of Lemma \ref{lemma:  commutation of y', y''},   we can write
 $$
 T_{\gamma, \bm a, \bm b, \bm c} =  g_\alpha\, ({y''_1}^{a_1} e_1)( {y''_3}^{\,a_3} e_3) \cdots 
 ( {y''_{2f-1}}^{a_{2f-1}} e_{2f-1})  g_\pi  {y''}^{\,\bm c}  
 {y'}^{\,\bm b} (g_\beta)^*
 $$
Now, using Lemma \ref{lemma:  reduction lemma for powers of y', y''},  we can write any such element as a linear combination elements  $T_{\gamma, \bm a', \bm b, \bm c}$   with the $a'_i$ in the desired interval.    Using the commutation relations again, we can also write
$$
 T_{\gamma, \bm a', \bm b, \bm c} =  g_\alpha\, {y''}^{\,\bm a'}\,  g_\pi   {y''}^{\,\bm c}  
 ( e_{2f-1} {y'_{2f-1}}^{b_{2f-1}})\cdots (e_3 {y'_3}^{\,b_3})(  e_1 {y'_1}^{\,b_1})   (g_\beta)^*,
$$
and using Lemma \ref{lemma:  reduction lemma for powers of y', y''},  we can write any such element as a linear combination of elements  $T_{\gamma, \bm a', \bm b', \bm c}$  with the $b'_i$ in the desired interval.   Finally,  $e_{2f-1}$  commutes with
$g_\pi$ and with all $y''_{2f + j}$.   Using  $ e_{2f-1} p_0(y''_{2f + j}) = 0$,  we can reduce any
$
 T_{\gamma, \bm a', \bm b', \bm c}$  to a linear combination  elements $
 T_{\gamma, \bm a', \bm b', \bm c'}$  with the $c'_i$ in the desired interval.
 
 It follows that $\bmw{n}$ is spanned by elements $T_{\gamma, \bm c}$,  where $\gamma$ is a permutation diagram and $0 \le c_i \le r$,  and by elements $T_{\gamma, \bm a, \bm b, \bm c}$ where
 $\gamma$ is a Brauer diagram with at least $2$ horizontal strands and 
  $0 \le a_i, b_i , c_i \le d$.   Moreover, the latter set spans
$\bmw{n} e_{n-1}  \bmw{n}$.   
 \end{proof}  
 
 \subsection{A spanning set for $\degbmw{n, S, r}$}  \label{subsection: spanning set for degenerate algebra}
 \def\gr{{\rm gr}}
 In this section, write $\degbmw{n}$  for
 $\degbmw{n, S, r}(\Omega;  u_1, \dots, u_r)$.

 Consider first the
 free non-commutative polynomial algebra in the generators  $\{s_i,e_i, y_j : \ 1\le i<n \text{ and }1\le j\le n \} $.      Assign degrees to the generators,  $\deg(e_i) = \deg(s_i) = 0$,  $\deg(y_j) = 1$.    This makes the
 non-commutative polynomial algebra into a graded algebra.    As the homomorphic image of a graded algebra is a filtered algebra,   the degenerate cyclotomic BMW algebra $\degbmw{n}$ is filtered by degree, as is the ideal   $\degbmw{n} e_{n-1} \degbmw{n} $.    Let  $\mathcal G = \gr(\degbmw{n})$ denote the associated graded algebra.   We will write $e_i, s_i,   y_j$  also for the images of these elements in $\mathcal G$.   
 
 Note that  $(\degbmw{n})_0$,  the degree zero part of $\degbmw{n}$, is the unital subalgebra generated by $\{s_i,e_i : \ 1\le i<n \}$.  The canonical map from $\degbmw{n}$ to $\mathcal G$ restricts to an algebra isomorphism from $(\degbmw{n})_0$ to $\mathcal G_0$.

 To produce a spanning set in the ideal  $\degbmw{n, S, r} e_{n-1} \degbmw{n, S, r} $, it suffices to produce a spanning set in   $\mathcal G e_{n-1} \mathcal G$.

 \begin{lemma} \label{lemma:  reduction lemma for powers of y in G}
In $\mathcal G$, we have
  $p_0(y_j)\, e_i = e_i  \,p_0(y_j) = 0$  for all $j  \ne  i+ 1$.  
\end{lemma}

\begin{proof}  In $\mathcal G$,  the elements $y_i$  become conjugate,  $s_i y_i s_i = y_{i+1}$.  
It follows that the proof of Lemma \ref{lemma:  reduction lemma for powers of y', y''} carries over  unchanged (replacing $y'_j$ with $y_j$  and $g_i$  with $s_i$ everywhere).   
\end{proof}
 
 We have already discussed the
 surjective homomorphism from the Brauer algebra $\mathcal B_n(\omega_0)$  with parameter 
 $\omega_0$  to $(\degbmw{n})_0$  taking $s_i \mapsto s_i$  and $e_i \mapsto e_i$;
 see the discussion just before Proposition \ref{proposition: spanning by r regular monomials}.     For a Brauer diagram $\gamma$,  we will also write $\gamma$ for the image of $\gamma$ in $(\degbmw{n})_0$  and in $\mathcal G_0$.   According to Proposition \ref{proposition: spanning by r regular monomials},   $\degbmw{n}$  is spanned by the set of $r$--regular monomials
\begin{equation} \label{spanning elements of N_n 2}
 y^{\bm p}  \gamma  y^{\bm q}.
 \end{equation}
   Furthermore, the ideal $\degbmw{n} e_{n-1} \degbmw{n}$ is spanned by those elements $y^{\bm p}  \gamma  y^{\bm q}$ such that $\gamma$ has $2f >0 $ horizontal strands.  
 
 If $\gamma$ is a permutation diagram, then $p_i = 0$ for all $i$ and
 $$
  y^{\bm p}  \gamma  y^{\bm q} =  \gamma  y^{\bm q}  :=  T_{\gamma, \bm q}.
 $$
 
 If $\gamma$ is not a permutation diagram, then using the factorization of $\gamma$ in Equation (\ref{equation:  factorization of Brauer diagram}), and using
 $s_i y_i = y_{i+1} s_i$  in $\mathcal G$, the image of the element (\ref{spanning elements of N_n}) in $\mathcal G$ can be written as
 \begin{equation}
 \begin{aligned}
 y^{\bm p}  \gamma  y^{\bm q}  &= y^{\bm p}   \alpha (e_1 e_3 \cdots e_{2f-1})  \pi   \beta\inv  y^{\bm q}  \\
 &=    \alpha  y^{\alpha\inv(\bm  p)} (e_1 e_3 \cdots e_{2f-1})  \pi  y^{ \beta\inv(\bm q)}  \beta\inv ,
 \end{aligned}
 \end{equation}
 where $  y^{\alpha\inv(\bm  p)} =   {y_1}^{p_{\alpha(1)}} \cdots {y_n}^{p_{\alpha(n)}}$ and
  $  y^{\beta\inv(\bm  p)} =   {y_1}^{p_{\beta(1)}} \cdots {y_n}^{p_{\beta(n)}}$.   Taking into account the restrictions on
  $\bm p$ and $\bm q$, this can be written in the form
  \begin{equation}
   T_{\gamma, \bm a, \bm b, \bm c}   =     \alpha  y^{\bm a} (e_1 e_3 \cdots e_{2f-1})  \pi  y^{\bm c}  y^{\bm b}  \beta\inv 
  \end{equation}
  where $$ y^{\bm a} = {y_1}^{a_1} {y_3}^{a_3}\cdots  {y_{2f-1}}^{a_{2f-1}},$$
  $$y^{\bm c} =  \prod_{2f+1 \le j \le n} y_j^{c_j},$$  and 
  $$ y^{\bm b} = {y_1}^{b_1} {y_3}^{b_3}\cdots  {y_{2f-1}}^{b_{2f-1}}.$$

  \begin{lemma} \label{lemma:  spanning set for ideal in degenerate case}
  The ideal \ $\degbmw{n} e_{n-1} \degbmw{n}$    is  spanned by a set of  $d^n b'(n)$ elements.   The algebra  $\degbmw{n}$  is spanned by a set of 
  $d^n b'(n) + r^n n!$ elements. 
\end{lemma}
 
 \begin{proof}  It is enough to work instead in the associated graded algebra $\mathcal G$.    We have that 
  $\mathcal G$ is spanned by the elements $T_{\gamma, \bm c}$,  where $\gamma$ is a permutation diagram and
  $0 \le c_i \le r-1$ for all $i$,   and by the elements
  $T_{\gamma, \bm a, \bm b, \bm c}$    where $\gamma$ is a Brauer diagram with at least 2 horizontal strands.
  The argument of Lemma \ref{lemma:  spanning set for the ideal in cyclotomic algebra}, with Lemma \ref{lemma:  reduction lemma for powers of y', y''} replaced by Lemma \ref{lemma:  reduction lemma for powers of y in G},  shows that any $T_{\gamma, \bm a, \bm b, \bm c}$,  where $\gamma$ has horizontal strands,  can be written as a linear combination of elements 
  $T_{\gamma, \bm a', \bm b', \bm c'}$, with $0 \le a_i, b_i, c_i \le d-1$.  Moreover, the latter set of elements spans
  $\mathcal G e_{n-1} \mathcal G$. 
  \end{proof}
 
 \subsection{Freeness of  $\both{n, S, r}$}  Let us recall from Lemma \ref{lemma: homomorphism from A n r to A n d}  that there is a surjective algebra homomorphism 
$\theta : 
  \both{n, S, r}(u_1, \dots, u_r) \to \both{n, S, d}(u_1, \dots, u_d)$ and that 
 $\theta$ maps the ideal generated by $e_{n-1}$ in   $\both{n, S, r}(u_1, \dots, u_r)$ onto the ideal generated by $e_{n-1}$ in  $\both{n, S, d}(u_1, \dots, u_d)$.

 \begin{proposition}  \label{proposition:  isomorphism for ideal in d admissible case}
  $\theta$ induces an isomorphism from the ideal generated by $e_{n-1}$ in    $\both{n, S, r}(u_1, \dots, u_r)$ onto the ideal generated by $e_{n-1}$ in  $\both{n, S, d}(u_1, \dots, u_d)$. 
  \end{proposition}
  
  \begin{proof}Write $\langle e_{n-1}\rangle_r$  for the ideal generated by $e_{n-1}$ in   $\both{n, S, r}(u_1, \dots, u_r)$ and
   $\langle e_{n-1}\rangle_d$  for the ideal generated by $e_{n-1}$ in   $\both{n, S, d}(u_1, \dots, u_d)$. 
  
  The parameters of $\both{n, S, d}(u_1, \dots, u_d)$ are admissible, by Lemma \ref{lemma:   semi--admissibility implies admissibility in fewer variables}.  Hence,   we know that  \break $\both{n, S, d}(u_1, \dots, u_d)$  is a free $S$ module of rank
  $d^n (2n-1)!!$,  and    $\langle e_{n-1}\rangle_d$  is free of rank
  $$
  d^n ((2n-1)!! - n!) =  d^n b'(n),
  $$
  where $b'(n)$ denotes the number of Brauer diagrams on $n$ strands with at least one horizontal strand. 
  We know that    $\langle e_{n-1}\rangle_r$ has a spanning set of the same cardinality by Lemmas \ref{lemma:  spanning set for the ideal in cyclotomic algebra} and \ref{lemma:  spanning set for ideal in degenerate case}.
  Therefore, $\theta :  \langle e_{n-1}\rangle_r \to \langle e_{n-1}\rangle_d$ is an isomorphism.  (In fact, if $\mathbb B$ is spanning set of $ \langle e_{n-1}\rangle_r$ of cardinality $d^n b'(n)$, then $\theta(\mathbb B)$ spans $\langle e_{n-1}\rangle_d$.  Since $S$ is an integral domain and $\langle e_{n-1}\rangle_d$ is free over $S$ with a basis of the same cardinality, it follows that $\theta(\mathbb B)$ is a basis of $\langle e_{n-1}\rangle_d$.  Therefore, $\mathbb B$ is a basis of  $ \langle e_{n-1}\rangle_r$, and $\theta$ is an isomorphism.)
  \end{proof} 
 
   \begin{theorem}   For all $n \ge 0$,   $\both{n, S,r}$ is a free $S$--module of rank $d^n b'(n) + r^n n!$, 
   and $\both{n, S,r}$  imbeds in $\both{n+1, S,r}$.   
   \end{theorem}
   
  \begin{proof}    The ideal  $\langle e_{n-1}\rangle_r$ is free of rank $d^n b'(n)$, by Proposition \ref{proposition:  isomorphism for ideal in d admissible case}, and the quotient
  $\both{n, S, r}/\langle e_{n-1}\rangle_r$ is isomorphic to the cyclotomic Hecke algebra or degenerate cyclotomic Hecke algebra, which is free of rank $r^n n$!.  Therefore,   $\both{n, S, r}$ is free of rank $d^n b'(n) + r^n n!$.
  
  We have given  spanning sets of the same cardinality in Lemmas \ref{lemma:  spanning set for the ideal in cyclotomic algebra} and \ref{lemma:  spanning set for ideal in degenerate case}, and hence those spanning sets are actually $S$--bases.  It is straightforward to check that the homomorphism from $\both{n, S ,r}$ to $\both{n+1, S ,r}$ taking
  generators to generators maps the given basis of  $\both{n, S ,r}$ injectively into the basis of $\both{n+1, S ,r}$.  Therefore the map is injective. 
  \end{proof} 
   
 \section{Cellularity}  \label{section: cellularity}

 The following is a slight weakening of the original definition of cellularity from Graham and Lehrer~\cite{Graham-Lehrer-cellular}.     
 
 \begin{definition}[\cite{Graham-Lehrer-cellular}]  \label{gl cell}  Let $R$ be an integral domain and $A$ a unital $R$--algebra.  A {\em cell datum} for $A$ consists of  an algebra involution $*$ of $A$; a partially ordered set $(\Lambda, \ge)$ and 
for each $\la \in \Lambda$  a set $\mathcal T(\lambda)$;  and   a subset $
\mathcal C = \{ c_{s, t}^\la :  \la \in \Lambda \text{ and }  s, t \in \mathcal T(\la)\} \subseteq A$; 
with the following properties:
\begin{enumerate}
\item  \label{gl cell item 1} 
$\mathcal C$ is an $R$--basis of $A$.
\item   \label{gl cell mult rule} For each $\la \in \Lambda$,  let $\breve A^\la$  be the span of the  $c_{s, t}^\mu$  with
$\mu > \la$.   Given $\la \in \Lambda$,  $s \in \mathcal T(\la)$, and $a \in A$,   there exist coefficients 
$r_v^s( a) \in R$ such that for all $t \in \mathcal T(\la)$:
$$
a c_{s, t}^\la  \equiv \sum_v r_v^s(a)  c_{v, t}^\la  \mod  \breve A^\la.
$$
\item  \label{gl cell conjugates} 
$(c_{s, t}^\la)^* \equiv c_{t, s}^\la   \mod  \breve A^\la$ for all $\la\in \Lambda$ and, $s, t \in \mathcal T(\lambda)$.

\end{enumerate}
$A$ is said to be a {\em cellular algebra} if it has a  cell datum.  
\end{definition}

For brevity,  we will write that  $(\mathcal C, \La)$ is a cellular basis of $A$.  
In the original definition in 
~\cite{Graham-Lehrer-cellular} it is required that  $(c_{s, t}^\la)^* = c_{t, s}^\la$.    All the conclusions of
~\cite{Graham-Lehrer-cellular} remain valid with the weaker definition, and, in fact, the two definitions are equivalent if $2$ is invertible in $R$.   The main advantage of the weaker definition is that it allows a graceful treatment of extensions.

\begin{definition} Let $A$ be an algebra with involution and let $J$ be a $*$--invariant ideal.
Say that  $J$ is a {\em cellular ideal}  if it satisfies the axioms for a  cellular algebra (except for being unital) with cellular basis 
$$ \{ c_{s, t}^\la :  \la \in \Lambda_J \text{ and }  s, t \in \mathcal T(\la)\} \subseteq J$$
and we have, as in point \refc{gl cell mult rule} of the definition of cellularity, 
$$
a c_{s, t}^\la  \equiv \sum_v r_v^s(a)  c_{v, t}^\la  \mod  \breve J^\la
$$
not only for $a \in J$ but also for $a \in A$.
\end{definition}

\begin{lemma}  \label{remark on extensions of cellular algebras} (On extensions of cellular algebras.) 
If $J$ is a cellular ideal in $A$, and 
$H = A/J$ is  cellular  (with respect to the involution induced from the involution on $A$), then $A$ is cellular.  
\end{lemma}

\begin{proof}
 Let $(\La_J, \ge)$ be the partially ordered set in the cell datum for $J$ and $\mathcal C_J$ the cellular basis.
Let $(\La_H, \ge)$  be the partially ordered set in the cell datum for $H$ and $\{\bar h_{u, v}^\mu\}$ the cellular basis.  Let  $\La = \La_J \cup \La_H$,  with partial order agreeing with the original partial orders on $\La_J$ and on $\La_H$ and with $\la > \mu$ if $\la \in \La_J$ and $\mu \in \La_H$.
A cellular basis of $A$ is $\mathcal C_J \cup \{h_{s, t}^\mu\}$, where $h_{s, t}^\mu$ is any lift of $\bar h_{s, t}^\mu$.
\end{proof}

 \begin{theorem}  Consider the sequence $\both{n, S, r}$  of cyclotomic or degenerate cyclotomic BMW algebras over an integral domain $S$.    Suppose that Assumption \ref{assumption: torsion free} holds.  
 Then 
 \begin{enumerate}
 \item  $\both{n, S, r}$ imbeds in $\both{n+1, S, r}$  for all $n \ge 0$.
 \item   $\both{n, S, r}$ is a cellular algebra.
 \end{enumerate}
 \end{theorem}
 
 \begin{proof}   In the case $e_1 = 0$ in $\both{2, S, r}$,  the cyclotomic or degenerate cyclotomic BMW algebras reduce to cyclotomic or degenerate cyclotomic Hecke algebras; in this case the results are known. 
 If the parameters are admissible, these results are obtained in the papers  cited in the introduction.   
 
 \def\mf{\mathfrak}
It  remains to verify the results in the semi--admissible case.  
We already have shown in the semi--admissible case that $\both{n, S, r}$ is a free $S$ module, and that 
$\both{n, S, r}$  imbeds in $\both{n+1, S, r}$.
Adopt the notation and conventions of Section \ref{section: semi admissibility}.   We know that  
$\both{n, S, d}(u_1, \dots, u_d)$   has admissible parameters by Lemma \ref{lemma:   semi--admissibility implies admissibility in fewer variables}, and therefore is a cellular algebra
by the papers cited in the introduction.  Moreover, $\langle e_{n-1} \rangle_d$ is a cellular ideal in $\both{n, S, d}(u_1, \dots, u_d)$.    It follows that $\langle e_{n-1} \rangle_r$ is a cellular ideal in $\both{n, S, r}$, with cellular basis
$\{\theta\inv(c_{\mf s, \mf t}^\la)\}$,   where $\{c_{\mf s, \mf t}^\la\}$ is a cellular basis of $\langle e_{n-1} \rangle_d$.  
The crucial point regarding the expansion of $a \theta\inv(c_{\mf s, \mf t}^\la)$ in terms of basis elements, for 
$a \in \both{n, S, r}$  follows because  $a \theta\inv(c_{\mf s, \mf t}^\la) = \theta\inv(\theta(a) c_{\mf s, \mf t}^\la)$.  

Since $\both{n, S, r}/ \langle e_{n-1} \rangle_r$ is isomorphic  to the cyclotomic Hecke algebra, or degenerate cyclotomic Hecke algebra, which is cellular, it follows from Lemma \ref{remark on extensions of cellular algebras} that $\both{n, S, r}$ is cellular. 
 \end{proof}
 
 \begin{corollary}  Any cyclotomic or degenerate cyclotomic BMW algebra over a field is cellular. 
 \end{corollary}
 
 \begin{proof} In case the ground ring is a field, Assumption \ref{assumption: torsion free} holds automatically.  
 \end{proof}
 
 \begin{corollary}  Let $F$ be an algebraically closed field and consider an 
 affine (resp.\   degenerate affine)  BMW algebra $\affboth{n, F}$ over $F$.   
    Let $M$ be a simple finite dimensional     $\affboth{n, F}$--module.   If $e_1 M = 0$,  then $M$ factors through a cyclotomic (resp.\  degenerate cyclotomic) Hecke algebra.  If $e_1 M \ne 0$,  then $M$ factors through  
    cyclotomic (resp.\  degenerate cyclotomic) BMW algebra with admissible parameters.  
 \end{corollary}
 
 \begin{proof}   In the degenerate case, this result is contained in~\cite{ariki-mathas-rui},  Theorem 7.19 and Proposition 3.11 (but with the  hypothesis that the characteristic of the field is $\ne 2$.)
 
 Because the field is algebraically closed,  the minimal polynomial of $y_1$ on $M$ factors
 over $F$.    Hence $M$ factors through some cyclotomic (resp.\  degenerate cyclotomic) BMW algebra.  If $e_1 M = 0$, then $M$ factors through the corresponding cyclotomic (resp.\  degenerate cyclotomic) Hecke algebra.  If $e_1 M \ne 0$,
 then the parameters of the cyclotomic (resp.\  degenerate cyclotomic) BMW algebra must be either admissible or semi--admissible. 
 
 Let us assume a cyclotomic (resp.\  degenerate cyclotomic) BMW algebra $\both{n, F, r}  = 
 \both{n F, r}(u_1, \dots, u_r)$ with $d$--semi--admissible parameters  ($d < r$).   Then $M$ is the simple head of a cell module
 $\Delta^\la$, and since $e_1 M \ne 0$,   the cell module belongs to the ideal $\langle e_{n-1} \rangle_r$.  
   But the cell modules belonging to $\langle e_{n-1} \rangle_r$ factor through
 $\theta : \both{n F, r}(u_1, \dots, u_r) \to \both{n F, d}(u_1, \dots, u_d)$,  and the latter algebra has admissible parameters.
\end{proof}
 
The following proposition depends only on the material in this paper up through Lemma \ref {lemma:   semi--admissibility implies admissibility in fewer variables}.

\begin{proposition} \label{proposition:  existence of fd module with e1 not zero}
 Let $F$ be an algebraically closed field and consider an 
 affine (resp.\   degenerate affine)  BMW algebra $\affboth{n, F}$ over $F$.     The following are equivalent:
 \begin{enumerate}
 \item  \label{proposition:  existence of fd module with e1 not zero item 1}
 There exist $r >0$  and $u_1, \dots, u_r \in F$  such that the parameters of  $\affboth{n, F}$ together with
 $u_1, \dots, u_r$  are admissible.
  \item  \label{proposition:  existence of fd module with e1 not zero item 2}
   $\affboth{n, F}$  admits a finite dimensional module on which $e_1$ is non--zero.
 \end{enumerate}
\end{proposition} 

\begin{proof}  If \refc{proposition:  existence of fd module with e1 not zero item 1} holds, then $\both{n, F, r}(u_1, \dots, u_r)$ is a finite dimensional $\affboth{n, F}$ module on which $e_1 \ne 0$.   If \refc{proposition:  existence of fd module with e1 not zero item 2} holds,  let $u_1, \dots, u_r$ be the roots of the minimal polynomial of $y_1$ acting on $M$.   The module $M$ factors through the cyclotomic algebra $\both{n, F, r}(u_1, \dots, u_r)$.  Since $e_1 M \ne 0$,  it follows that $e_1 \ne 0$ in $\both{n, F, r}(u_1, \dots, u_r)$ and hence also in 
$\both{2, F, r}(u_1, \dots, u_r)$.    Since $F$ is a field, Assumption \ref{assumption: torsion free} holds for $\both{2, F, r}(u_1, \dots, u_r)$.  Therefore, the parameters of $\both{2, F, r}(u_1, \dots, u_r)$ are either admissible or 
$d$--semi--admissible for some $d$ with $0 < d < r$.  In the latter case, after renumbering the roots $u_i$, 
$\both{2, F, d}(u_1, \dots, u_d)$ has admissible parameters, by Lemma \ref {lemma:   semi--admissibility implies admissibility in fewer variables}.  Thus \refc{proposition:  existence of fd module with e1 not zero item 1} holds.  
\end{proof}

 \section{Rationality of parameters for affine algebras}
 
 \subsection{Rationality of parameters for degenerate affine BMW algebras}
 Ariki, Mathas, and Rui call the parameter set $\Omega$ of a degenerate affine (or cyclotomic) BMW algebra {\em rational} if the generating function $\sum_{a \ge 0} \omega_a t^{-a}$ is a rational function. 
 They prove the following theorem, under the additional hypothesis that the characteristic of the field is different from 2.
 
 \begin{theorem}  \label{theorem:  ARM theorem on deg. affine algebras with rational coefficient sequence}
 Consider the degenerate affine BMW algebra $\affdegbmw{n}$, $n \ge 2$,   over an algebraically closed field $F$, with parameters  $\Omega = (\omega_a)_{a \ge 0}$. Suppose that $e_1 \ne 0$ in $\affdegbmw{n}.$  The following are equivalent.
 \begin{enumerate}
  \item  \label{armthm item: rational function}
  The generating function  $\sum_{a \ge 0} \omega_a t^{-a}$  is a rational function in $F(t)$.  
 \item  \label{armthm item: recursion}
 $\Omega$ satisfies a linear homogeneous recursion; i.e.\    there exist $r>0$, $N \ge 0$ and $a_0, a_1, \dots, a_{r-1} \in F$  such that
$\omega_{r + \ell} +    \sum_{j = 0}^{r-1}  a_j  \omega_{j + \ell}  = 0,$ for all  $\ell \ge N$.
\item   \label{armthm item: global recursion}
There exist $r>0$  and $a_0, a_1, \dots, a_{r-1} \in F$  such that 
$\omega_{r + \ell} +    \sum_{j = 0}^{r-1}  a_j  \omega_{j + \ell}  = 0,$ for all  $\ell \ge 0$.
 \item  \label{armthm item: admissible}
 There exist $r >0$  and $u_1, \dots, u_r \in F$  such that  the parameters  \ $\Omega$  and  $u_1, \dots, u_r$ are admissible.
 \item   \label{armthm item: fd repns}
 $\affdegbmw{n}$  admits a finite dimensional module on which $e_1$ is non--zero.
 \end{enumerate}
 \end{theorem}
 
 \begin{proof}  \refc{armthm item: rational function} $\iff$ \refc{armthm item: recursion} $\Longleftarrow$ \refc{armthm item: global recursion} is easy, and  \refc{armthm item: global recursion} $\Longleftarrow \refc{armthm item: admissible}$ holds by Lemma \ref{lemma: periodicity of omegas when e1 is torsion free}. 
  Proposition \ref{proposition:  existence of fd module with e1 not zero} gives  \refc{armthm item: admissible} $\iff$ \refc{armthm item: fd repns}.  
 The implication \refc{armthm item: rational function} $\implies$ \refc{armthm item: admissible}  is proved in~\cite{ariki-mathas-rui}, Proposition 3.11,  under the assumption that the characteristic of the field is not equal to 2.  So it remains only to prove this implication for a field of characteristic 2.  This will be done with the aid of two  lemmas.
  \end{proof}
  
  \begin{lemma}   \label{lemma:  ARM weak admissibility conditions}
  Consider the degenerate affine BMW algebra $\affdegbmw{2, S}$  over a ring  $S$, with parameters  $\Omega = (\omega_a)_{a \ge 0}$. Suppose that $e_1$  is not a torsion element over $S$.  Then:
  \begin{enumerate}
  \item   \label{lemma:  ARM weak admissibility conditions item 1}
  $2 \omega_{2a+1}= -\omega_{2a}
            +\sum_{b=1}^{2a+1}(-1)^{b-1}\omega_{b-1}\omega_{2a+1-b}$ \ for $a \ge 0$.
  \item   \label{lemma:  ARM weak admissibility conditions item 2}
  If the characteristic of $S$ is $2$,  then  $\omega_{2a} = \omega_a^2$ \  for $a \ge 0$.  
  \end{enumerate}
  \end{lemma}
  
  \begin{proof} Part \refc{lemma:  ARM weak admissibility conditions item 1} is~\cite{ariki-mathas-rui},  Corollary 2.4.     If the characteristic is $2$,  then
  the equation in part \refc{lemma:  ARM weak admissibility conditions item 1} simplifies to  $\omega_{2a} = \omega_a^2$.  
  \end{proof}
  
  \vbox{
  The proof of the following lemma was suggested by Kevin Buzzard, via {\tt mathoverflow.net}.
  
  \begin{lemma}  \label{lemma:  Kevin Buzzard lemma}
  Let $F$ be an algebraically closed field of characteristic $2$.  Suppose that\ 
  $\ \Omega = (\omega_a)_{a \ge 0}$ satisfies a linear homogeneous recursion, as in Theorem \ref{theorem:  ARM theorem on deg. affine algebras with rational coefficient sequence} 
  \refc{armthm item: recursion}
   and
  $\omega_{2a} = \omega_a^2$  for $a \ge 0$.
 Then there exist distinct $u_1, \dots, u_d \in F$  such that $\omega_a = \sum_{i = 1}^d u_i^a$ for all 
 $a \ge 1$, and $\omega_0 \in  \{0, 1\}$.    
\end{lemma}
}

\begin{proof}  Our assumptions include $\omega_0 = \omega_0^2$.  Thus $\omega_0 \in  \{0, 1\}$.
Let $v_1, \dots, v_m$  be the distinct roots of the characteristic polynomial of the linear recursion relation satisfied by $\Omega$.    Then there exist polynomials $h_1, \dots, h_m$ such that
$\omega_a = \sum_{i = 1} ^m  h_i(a) v_i^a $  for $a \ge N$.    Let $\alpha_i$ be the constant term of 
$h_i$ for each $i$.   Since $\char(F) = 2$,  we have $h_i(2a) = \alpha_i$  for all $a$.  
For $a \ge N$,  
\begin{equation} \label{equation: linear recurrance in characteristic two 1}
 \sum_i  \alpha_i v_i^{4a} = \omega_{4 a}   =  \omega_{2a}^2 = \sum_i \alpha_i^2 v_i^{4a}.
\end{equation}
Because the characteristic of $F$ is $2$,  each element has a unique  $2^k$--th root for all $k \ge 1$;  in particular all the $v_i^4$ are distinct, so Equation (\ref{equation: linear recurrance in characteristic two 1}) implies that $\alpha_i^2 = \alpha_i$ for all $i$, i.e.\  $\alpha_i \in \{0, 1\}$.    Let $u_1, \dots, u_d$ be the list of those $v_j$ such that $\alpha_j = 1$.  Then we have $\omega_{2a} = \sum_i u_i^{2a}$ for
$a \ge N$.   For an arbitrary $a \ge 1$,  chose $k$ such that $2^{k-1} a \ge N$.  Then
$\omega_a$ is the unique $2^k$--th root of $\omega_{2^k a} = \sum_i u_i^{2^k a}$, namely
$\omega_a = \sum_i u_i^a$.  
\end{proof}

\noindent{\em Conclusion of the proof of Theorem \ref{theorem:  ARM theorem on deg. affine algebras with rational coefficient sequence}.}  Let us prove 
\refc{armthm item: rational function} $\implies$ \refc{armthm item: admissible} 
 when the characteristic of the field is $2$.   Since the ground ring is a field and $e_1 \ne 0$,  we have $\omega_{2a} = \omega_a^2$ for
$a \ge 0$, by Lemma \ref{lemma:  ARM weak admissibility conditions}.   Hence, by Lemma \ref{lemma:  Kevin Buzzard lemma},   there exist $u_1, \dots, u_d \in F$  such that $$\omega_a = p_a(u_1, \dots, u_d) = p_a(u_1, \dots, u_d, 0)$$ for $a \ge 1$  and $\omega_0 \in \{0, 1\}$.    Using Example 
\ref{example:  evaluation of eta's when characteristic is 2} and Definition \ref{definition: u admissibility for degenerate case},   $\Omega$ is either $(u_1, \dots, u_d)$--admissible or $(u_1, \dots, u_d, 0)$--admissible, so by Theorem \ref{theorem:  equivalent admissibility conditions for degenerate algebras}, on equivalent conditions for admissibility, condition \refc{armthm item: admissible}  holds.  \qed
 
 \begin{corollary}[Rui and Si \cite{rui-si-degenerate}]   Assume $\char(F) \ne 2$.  
 The conditions of  Theorem \ref{theorem:  ARM theorem on deg. affine algebras with rational coefficient sequence} are equivalent to the existence of a simple finite dimensional module on which $e_1$ is non--zero, as long as \ $\Omega$ is not the zero sequence or
 $n \ne 2$.  
 \end{corollary}

 \begin{proof}  By the results of~\cite{rui-si-degenerate},  a degenerate cyclotomic BMW algebra
 $\degbmw{n, F, r}(\Omega;  u_1, \dots, u_r)$  with admissible parameters has a simple module on which $e_1$ is non--zero, as long as $\Omega$ is not the zero sequence or $n \ne 2$.  
 (Rui and Si assumed $\char(F) \ne 2$,  and I have not checked whether their results remain valid in characteristic $2$.) 
 \end{proof}
 
 \subsection{Rationality of parameters for affine BMW algebras}  We are going to obtain a result analogous to Theorem \ref{theorem:  ARM theorem on deg. affine algebras with rational coefficient sequence} for the affine BMW algebras.

 \begin{lemma}  \label{lemma: recursion for omegas with negative index and generating function identity}
 Consider an affine BMW algebra
 $\abmw{n, S}$ with parameters $\rho$, $q$, and $\Omega = (\omega_a)_{a \ge 0}$.   
  \begin{enumerate}
  \item \label{lemma: recursion for omegas with negative index and generating function identity item 1}
 There exist elements $\omega_{-a} \in S$  such that
 $e_1 y_1^{-a} e_1 = \omega_{-a} e_1$  for $a \ge 1$.  
 \item  \label{lemma: recursion for omegas with negative index and generating function identity item 2}
 Suppose that  $e_1$ is not a torsion element over $S$.  Then:
\begin{equation} \label{equation:  recursion for the omegas with negative index}
-\omega_a + \omega_{-a} + \rho (q - q\inv) \sum_{i = 1}^a (\omega_{a-i} \omega_{-i} - \omega_{a - 2i}) = 0
\quad \text{for }  a \ge 1.
\end{equation}
\item \label{lemma: recursion for omegas with negative index and generating function identity item 3}
Suppose that $S$ is an integral domain, that $q - q\inv \ne 0$, and that $e_1$ is not a torsion element over $S$.  Then:  
\begin{equation}  \label{equation:  identity for wplus and wminus}
\begin{aligned}
\left [ \sum_{a \ge 0}  \omega_a t^{-a} - \frac{t^2}{t^2-1} + \frac{\rho\inv}{q - q\inv} \right ] &
\left [   \sum_{b \ge 1} \omega_{-b} t^{-b}  - \frac{1}{t^2-1} - \frac{\rho\inv}{q - q\inv} \right ] \\ &=  
   \frac{t^2}{(t^2 -1)^2}  - \frac{1}{(q - q\inv)^2}.  
\end{aligned}
\end{equation} 
\end{enumerate}
 \end{lemma}
 
 \begin{proof}  Statement 
 \refc{lemma: recursion for omegas with negative index and generating function identity item 1}
  is from~\cite{GH1}, Corollary 3.13.   Statement  
  \refc{lemma: recursion for omegas with negative index and generating function identity item 2} is proved in ~\cite{rui-2008}, Lemma 2.17 and  (in a different but equivalent form) in ~\cite{GH1}, Corollary 3.13 and ~\cite{GH2}, Lemma 2.6.
  The equation (\ref{equation:  identity for wplus and wminus}) appears as (2.30) in 
  ~\cite{rui-2008}.
   If $S$ is integral and  $q - q\inv \ne 0$,  then (\ref{equation:  recursion for the omegas with negative index})  is equivalent to
 (\ref{equation:  identity for wplus and wminus}).  To see this, expand the left side of 
  (\ref{equation:  identity for wplus and wminus}) and isolate the coefficient of $t^{-n}$ for each $n \ge 0$.    
  
  \end{proof}
 
\begin{remark}  The equivalence of (\ref{equation:  recursion for the omegas with negative index})  and   (\ref{equation:  identity for wplus and wminus})  seems to be implicit in ~\cite{rui-2008}.    The left side of  (\ref{equation:  identity for wplus and wminus})   can also be written as:
$$
\left [ \sum_{a \ge 0}  \omega_a t^{-a} - \frac{t^2}{t^2-1} + \frac{\rho\inv}{q - q\inv} \right ] 
\left [   \sum_{b \ge 0} \omega_{-b} t^{-b}  - \frac{t^2}{t^2-1} - \frac{\rho}{q - q\inv} \right ].
$$
Ram et. al. ~\cite{Ram-Daugherty-Virk} have given an interesting non-inductive direct proof of (\ref{equation:  identity for wplus and wminus}).

 \end{remark}
 
 \begin{lemma}  \label{lemma:  periodicity of the omegas in the affine BMW}
 Consider an cyclotomic  BMW algebra
 $\bmw{n, S, r}$ with parameters $\rho$, $q$, and $\Omega = (\omega_a)_{a \ge 0}$   and
 $u_1, \dots, u_r$.   Let $a_i$ be given by equation (\ref{equation:  relate a's to u's}).   If $e_1$ is not a torsion element over $S$, then  
 $
 \sum_{j = 0}^r  a_j  \omega_{j + \ell}  = 0$  for all $\ell  \in \Z$.
 \end{lemma}
 
 \begin{proof}  Same as the proof of Lemma \ref{lemma: periodicity of omegas when e1 is torsion free}. 
 \end{proof}
 
  \begin{lemma}  \label{lemma:  recursion relation in affine bmw and generating functions}
  Consider an affine BMW algebra
 $\abmw{n, F}$ with parameters $\rho$, $q$, and $\Omega = (\omega_a)_{a \ge 0}$ over a field $F$.     Suppose that there exist  $r>0$  and $a_0, a_1, \dots, a_{r-1} \in S$ such that $
\omega_{a + r} +  \sum_{j = 0}^{r-1}  a_j  \omega_{j + a}  = 0$  for all $a  \in \Z$.  Then  $w^+(t) = \sum_{a \ge 0}  \omega_a t^{-a} $  and $w^-(t) = \sum_{b \ge 1}  \omega_{-b} t^{-b} $ are rational functions in $F(t)$  and 
 $w^-(t) = - w^+(t\inv)$.  Moreover, $w^+(0) = 0$  and $w^+(\infty) = \omega_0$.  
 \end{lemma}
 
 \begin{proof}  Let $p(t) = t^r + \sum_{j = 0}^{r-1}  a_j t^j$.   Then one computes, using the recursion on
 $(\omega_a)_{a \in \Z}$,  that
 $p(t) w^+(t) = q_1(t)$,  where $q_1$ is an explicit polynomial of degree $\le r$.  Similarly,
 $p(t) w^-(t\inv) = q_2(t)$.  Using the recursion again, one sees that $q_1 = - q_2$.    The coefficient of 
 $t^r$ in $q_1(t)$ is $\omega_0$ and the constant term is zero;  this gives $w^+(0) = 0$  and $w^+(\infty) = \omega_0$. 
 
 \end{proof}

 \begin{theorem}  \label{theorem:  theorem on  affine bmw algebras with rational coefficient sequence}
  Consider an  affine BMW algebra $\affbmw{n}$ over an algebraically closed field $F$, with parameters  $\rho$, $q$, and  $\Omega = (\omega_a)_{a \ge 0}$. 
 Suppose that $e_1 \ne 0$ in $\affbmw{n}.$  Consider the following statements:
  \begin{enumerate}
   \item  \label{thm item: w+ equals w- at t inverse}
   $w^+(t) = \sum_{a \ge 0}  \omega_a t^{-a} $  and $w^-(t) = \sum_{b \ge 1}  \omega_{-b} t^{-b} $ are rational functions in $F(t)$  and 
 $w^-(t) = - w^+(t\inv)$.    Moreover, $w^+(t)$ does not have a pole at $0$ or at $\infty$.  
 \item   \label{thm item: linear recurrence over all of Z}
 There exist $r>0$  and $a_0, a_1, \dots, a_{r-1} \in F$  such that 
$\omega_{r + \ell} +    \sum_{j = 0}^{r-1}  a_j  \omega_{j + \ell}  = 0,$ for all  $\ell \in \Z$.
 \item   \label{thm item: all admissible parameters}
 There exist $r >0$  and $u_1, \dots, u_r \in F$  such that  the parameters  $\rho$, $q$,  
 $\Omega$, and  $u_1, \dots, u_r$ are admissible.   
   \item   \label{thm item: allow fd repn with e1 not zero}
   $\affbmw{n}$  admits a finite dimensional module on which $e_1$ is non--zero.
  \end{enumerate}
  The following implications hold:    
  $$
  \refc{thm item: w+ equals w- at t inverse}  \Longleftarrow \refc {thm item: linear recurrence over all of Z} \Longleftarrow \refc{thm item: all admissible parameters} \iff \refc{thm item: allow fd repn with e1 not zero}.
  $$ 
  If $q - q\inv \ne 0$,  then all the conditions  are equivalent.  
  \end{theorem}
  
  \begin{proof}   The implication $\refc{thm item: w+ equals w- at t inverse}  \Longleftarrow \refc {thm item: linear recurrence over all of Z}$ is from Lemma \ref{lemma:  recursion relation in affine bmw and generating functions}, and 
  $\refc {thm item: linear recurrence over all of Z} \Longleftarrow \refc{thm item: all admissible parameters}$  from Lemma \ref{lemma:  periodicity of the omegas in the affine BMW}.  The equivalence 
$ \refc{thm item: all admissible parameters} \iff \refc{thm item: allow fd repn with e1 not zero}$ comes from Proposition \ref{proposition:  existence of fd module with e1 not zero}.   
  
It remains to prove $\refc{thm item: w+ equals w- at t inverse} \implies \refc{thm item: all admissible parameters}$ if $q - q\inv \ne 0$.  
Assume \refc{thm item: w+ equals w- at t inverse}.  Since the ground ring is a field and $e_1$ is assumed to be non--zero,  (\ref{equation:  identity for wplus and wminus}) holds.   But by assumption, we have that $w^+(t) = \sum_{a \ge 0} \omega_a t^{-a}$  and $w^-(t) = \sum_{b \ge 1} \omega_{-b} t^{-b}$ are rational functions, and
$w^-(t) = -w^+(t\inv)$.    Substituting in (\ref{equation:  identity for wplus and wminus}),  and writing
$$
h(t) = -\frac{t^2}{t^2 - 1}  + \frac{\rho\inv}{q - q\inv}, 
$$

we have
\begin{equation} \label{equation:  constraint equation on w+}
- \left [ w^+(t) + h(t) \right ] \left [ w^+(t\inv ) + h(t\inv ) \right ]
=   \frac{t^2}{(t^2 -1)^2}  - (q - q\inv)^{-2}  .
\end{equation}
Define $$B(t) = (q-q\inv)\inv +  \frac{t}{t^2 -1} =  \frac{(t+q)(t - q\inv)}{(q - q\inv)(t^2 - 1)}.$$  
Note that 
\begin{equation} \label{equation: B of t equation}
-B(t) B(t\inv) =  \frac{t^2}{(t^2 -1)^2}  - (q - q\inv)^{-2}  .
\end{equation}
We can write $w^+(t)$ in the form
$$
w^+(t) = - {h(t)} + B(t) A_0 t^m \frac{\prod_{\ell = 1}^s (t u_\ell - 1)}{\prod_{j = 1}^r (t - v_j)}, 
$$
where $m \in \Z$, $A_0 \in F$,  no $u_\ell$ or $v_j$ is zero, and $u_\ell \ne v_j\inv$ for all $j, \ell$.  
Then, taking into account equations (\ref{equation:  constraint equation on w+}) and (\ref{equation: B of t equation}) we have
\begin{equation}
1 = A_0^2   \frac{\prod_{\ell = 1}^s (t u_\ell - 1) (t\inv u_\ell - 1)}{\prod_{j = 1}^r (t - v_j)(t\inv - v_j)}  
= A_0^2 (-t)^{r-s}  \frac{\prod_\ell u_\ell}{\prod_j v_j}\frac{\prod_{\ell = 1}^s (t  - u_\ell\inv) (t - u_\ell )}{\prod_{j = 1}^r (t - v_j\inv)(t - v_j)} .
\end{equation}
Considering the restrictions placed on the $u_\ell$ and $v_j$, we must have $r = s$,  $A_0^2 = 1$, and
the multisets $\{u_1, \dots, u_s\}$  and $\{v_1, \dots, v_s\}$ coincide.  Thus
\begin{equation} \label{equation: expression for w+}
w^+(t) = - {h(t)} + (-1)^\alpha  B(t) t^m  \prod_{j = 1}^s  \frac{tu_j - 1}{t - u_j},
\end{equation}
with $\alpha \in \{0, 1\}$  and $m \in \Z$.  Because $w^+$ does not have a pole at $0$ or $\infty$,  we have $m = 0$.   Using the definition of $h(t)$, we have finally
\begin{equation} \label{equation: expression for w+ 2}
w^+(t) = \frac{t^2}{t^2-1} - \rho\inv (q - q\inv)\inv + (-1)^\alpha  B(t)   \prod_{j = 1}^s  \frac{tu_j - 1}{t - u_j}.
\end{equation}
Moreover,  using $w^+(\infty) = \omega_0$, we obtain that
$$
(\omega_0 - 1)(q - q\inv) = -\rho\inv + (-1)^\alpha \textstyle \prod_j u_j,
$$
and  (\ref{equation:  basic relation in ground ring}) implies  that $\rho = (-1)^\alpha \prod_j u_j$.   Now there are four cases to consider, according to the parity of $\alpha$ and of $s$.

{\em Case 1,  $\alpha = 0$ and $s$ is odd.}   Then $\rho = \prod_j u_j$.    Comparing the expression (\ref{equation: expression for w+ 2})  for $w^+(t)$ with the formulas  (\ref{equation: formula for Z of t}) and (\ref{equation: formula for Z of t 2}) and Definition  \ref{definition: RX admissibility}, we see that the parameters $\rho$, $q$, $\Omega$ are
$(u_1, \dots, u_s)$--admissible.

{\em Case 2, $\alpha = 1$ and $s$ is odd.}  Then $\rho = -\prod_j u_j$.  Let $v = (u_1, \dots, u_s, -1, 1)$.
Then  
$$
\begin{aligned}
w^+(t) &= \frac{t^2}{t^2-1} - \rho\inv (q - q\inv)\inv - B(t) \, \left ({\textstyle \prod_{j = 1}^{s} u_j }\right )\prod_{j =1}^s 
\frac{t - u_j\inv}{t- u_j} \\
&= \frac{t^2}{t^2-1} - \rho\inv (q - q\inv)\inv  +  B(t)  \, \left ({\textstyle \prod_{j = 1}^{s+2} v_j }\right )
 \prod_{j =1}^{s+2} 
\frac{t - v_j\inv}{t- v_j},
\end{aligned}
$$
and $\rho = - \prod_{j = 1}^s u_j = \prod_{j = 1}^{s+2} v_j$.  
Again, comparing with the formulas of Section \ref{subsection: admissibility criterion of Rui and Xu}, we see that the parameters $\rho$, $q$, $\Omega$ are  $(u_1, \dots, u_s, -1, 1)$--admissible.  

{\em Case 3, $\alpha = 0$ and $s$ is even.}  Then $\rho = \prod_j u_j$.   
By a similar calculation as in Case 2, one checks that the parameters $\rho$, $q$, $\Omega$ are
$(u_1, \dots, u_s, 1)$--admissible.

{\em Case 4, $\alpha = 1$ and $r$ is even.}  Then $\rho = -\prod_j u_j$.   
By a similar calculation again, one checks that the parameters $\rho$, $q$, $\Omega$ are
$(u_1, \dots, u_s, -1)$--admissible.

In each of the four cases, there exists $r > 0$ and $v_1, \dots, v_r$ such that
$\rho$, $q$, $\Omega$ and $v_1, \dots, v_r$ satisfy the Rui-Xu criterion for admissibility. Thus we have shown $\refc{thm item: w+ equals w- at t inverse} \implies \refc{thm item: all admissible parameters}$ when $q - q\inv \ne 0$.  
  \end{proof}

  \begin{corollary}[Rui and Si~\cite{Rui-Si-Cyclotomic-II}] Assume $q - q\inv \ne 0$.
The conditions   of  Proposition \ref{theorem:  theorem on  affine bmw algebras with rational coefficient sequence}
  are equivalent to the existence of a simple finite dimensional module on which $e_1$ is non--zero, as long as \ $\Omega$ is not the zero sequence or
 $n \ne 2$.   
\end{corollary}

\begin{proof}  By the results of~\cite{Rui-Si-Cyclotomic-II},  a  cyclotomic BMW algebra
 $\bmw{n, S, r}(\rho, q, \Omega; u_1, \dots, u_r)$  with admissible parameters and $q - q\inv \ne 0$ has a simple module on which $e_1$ is non--zero, as long as $\Omega$ is not the zero sequence or $n \ne 2$.   
\end{proof}

\begin{conjecture} Theorem \ref{theorem:  theorem on  affine bmw algebras with rational coefficient sequence} remains valid when  $q - q\inv = 0$.    
\end{conjecture}

 \section{Construction of examples of semi--admissible parameters}
 Examples of cyclotomic (resp.\  degenerate cyclotomic)  BMW algebras with semi--admissible parameters can easily be constructed.  For the sake of clarity, we carry this out for the degenerate cyclotomic BMW algebras only;  non-degenerate cyclotomic BMW algebras with $q^2 \ne 1$  can be treated in a similar way, using the admissibility criterion of Rui and Xu~\cite{rui-2008}.     
 
 Let $S$ be an integral domain with $1/2 \in S$.  Take $0 < d < r$ and $u_1, \dots, u_r \in S$.    
 Assume that $u_i \ne \pm u_j$ for any $i, j$ and that $u_i \ne \pm 1/2$ for any $i$.  
 Let $p(u) = \prod_{1 \le j \le r} (u - u_j)$  and  $p_0(u) = \prod_{1 \le j \le d} (u - u_j)$.  
 Define
 $\omega_a$ for $a \ge 0$ via the $(u_1, \dots, u_d)$--admissibility criterion of~\cite{ariki-mathas-rui}, 
\begin{equation} \label{equation: u admissibility relations up to d}
\omega_a = q_{a+1}(u_1, \dots, u_d) + \frac{1}{2} (-1)^{d-1}q_a(u_1, \dots, u_d)  + \frac{1}{2} \delta_{a, 0}.
\end{equation}
By~\cite{ariki-mathas-rui}, this is equivalent to 
\begin{equation} \label{equation: section 6 first u admissibility equation}
\sum_{a \ge 0} \omega_a  u^{-a} =  1/2 - u + (u - (-1)^{d}/2) \prod_{j = 1}^d  \frac{u + u_j}{u - u_j}
\end{equation}
By the implication 
$\refc{theorem:  equivalent admissibility conditions for degenerate algebras item 4} \implies \refc{theorem:  equivalent admissibility conditions for degenerate algebras item 1}$ 
in Theorem \ref{theorem:  equivalent admissibility conditions for degenerate algebras}  (which is from~\cite{ariki-mathas-rui}),  the parameters
$\Omega = (\omega_a)_{a \ge 0}$  and $u_1, \dots, u_d$  are admissible; i.e.\ 
 the set $\{ e_1, y_1 e_1, \dots, y_1^{d-1} e_1\}$ is linearly independent over $S$ in   $\degbmw{2, S, d}(\Omega, u_1, \dots, u_d)$.
 
 Now consider $\degbmw{2, S, r}(\Omega; u_1, \dots, u_r)$.  Since we have an algebra map
 $$\theta : \degbmw{2, S, r}(\Omega; u_1, \dots, u_r) \to \degbmw{2, S, d}(\Omega; u_1, \dots, u_d),$$
 we have 
    $\{ e_1, y_1 e_1, \dots, y_1^{d-1} e_1\}$ is linearly independent over $S$ in $\degbmw{2, S, r}(\Omega; u_1, \dots, u_r)$.  Let $r'$  be maximal such that $\{ e_1, y_1 e_1, \dots, y_1^{r'-1} e_1\}$ is linearly independent in $\degbmw{2, S, r}(\Omega; u_1, \dots, u_r)$.  Then by the argument following Definition \ref{definition: semi admissible},  there is a subset $\{v_1, \dots, v_{r'}\}$  of 
    $\{u_1, \dots, u_r\}$  such that   $p_1(y_1) e_1 : =  \prod_{j = 1}^{r'} (y_1 - v_j)  e_1  = 0$, and
    $h(y_1) e_1 \ne 0$  for any polynomial $h$ of degree less than $r'$.   Now by Lemma
    \ref{lemma:   semi--admissibility implies admissibility in fewer variables}  and Theorem
    \ref{theorem:  equivalent admissibility conditions for degenerate algebras},  the set of parameters
    $\Omega, v_1, \dots, v_{r'}$  satisfies the $(v_1, \dots, v_{r'})$--admissibility conditions.  Hence we also have
\begin{equation} \label{equation: section 6 second u admissibility equation}
\sum_{a \ge 0} \omega_a  u^{-a} =  1/2 - u + (u - (-1)^{r'}/2) \prod_{j = 1}^{r' } \frac{u + v_j}{u - v_j}.
\end{equation}    
Comparing Equations (\ref{equation: section 6 first u admissibility equation}) and  (\ref{equation: section 6 second u admissibility equation}),  and taking into account the assumptions on 
$\{u_1, \dots, u_r\}$,  we conclude that $d = r'$  and $\{v_1, \dots, v_d\} = \{u_1, \dots, u_d\}$.   
 Thus the parameters $\Omega, u_1, \dots, u_r$  are $d$--semi--admissible.

\bibliographystyle{amsplain}
\bibliography{admissibility}

 \end{document}